\documentclass{amsart}
\usepackage{chngcntr}
\usepackage{hyperref}
\usepackage{apptools}
\usepackage{graphicx}
\usepackage{mathrsfs}
\usepackage{enumerate}
\usepackage[shortlabels]{enumitem}
\usepackage{ upgreek }
\usepackage[latin1]{inputenc} 
\usepackage{color}
\usepackage{amsthm,amssymb,verbatim}
\usepackage{enumerate}
\usepackage{ragged2e}
\usepackage{blindtext}
	\usepackage{pgfplots}

\newcommand{\rr}{{\mathbb R}}

\newcommand{\Sp}{{\mathbb S}}

\newcommand{\eps}{\varepsilon}

\newcommand{\set}[1]{\left\{#1\right\}}
\newcommand{\pare}[1]{\left(#1\right)}
\newcommand{\abs}[1]{\left|#1\right|}
\newcommand{\norm}[1]{\abs{\abs{#1}}}

\newcommand{\W}{{\mathcal W}}

\newcommand{\pI}[1]{\left <{#1}\right >}
\newcommand{\restri}[2]{\left.#1\right|_{#2}}
\newcommand\restr[2]{{
		\left.\kern-\nulldelimiterspace 
		#1 
		\vphantom{\big|} 
		\right|_{#2} 
}}

\newtheorem{theorem}{Theorem}[section]
\newtheorem{lemma}[theorem]{Lemma}
\newtheorem{proposition}[theorem]{Proposition}
\newtheorem{corollary}[theorem]{Corollary}
\newtheorem{definition}[theorem]{Definition}

\newtheorem{remark}[theorem]{Remark}

\usepackage{hyperref}
\usepackage{amscd}
\usepackage[ msc-links, backrefs]{amsrefs}
\usepackage[foot]{amsaddr}

\title[Maximum principle for $\gamma$-translators II]{Maximum Principles and Consequences for $\gamma$-translators in $\rr^{n+1}$ II}
\author{José Torres Santaella}
\address{Department of Mathematics and Statictits\\ Southern Maine University\\ Portland, ME, US.}
\email{jos.torres@maine.edu}
\begin{document}

	\maketitle
	
	\begin{abstract}
	This paper focuses on the translating solitons of fully nonlinear extrinsic curvature geometric flows in $\mathbb{R}^{n+1}$. We present a generalization of the Spruck-Xiao's and Spruck-Sun's convexity results for $1$-homogeneous convex/concave curvature functions, and further provide several characterizations of the family of grim reaper cylinders under curvature constraints.
	\end{abstract}

	\section{\textbf{Introduction}}
	Extrinsic geometric flows (EGFs) belong to the theory of geometric flows, which are powerful tools for investigating the deformation of geometric objects such as manifolds, maps, tensors, or geometric structures over time. The goal is to simplify a geometric quantity (usually by making it constant) and characterize the resulting deformation in terms of the shape, curvature and topology of the geometric object.
	\newline
	
	In the theory of EGFs, we study the evolution/deformation of an initial immersed hypersurface $\Sigma_0=F_0(\Sigma)$ in a Riemannian manifold (in this paper, $\rr^{n+1}$ with its Euclidean metric structure). The speed and direction of this deformation are determined by extrinsic geometric objects of the evolution $\Sigma_t=F(\Sigma,t)$, where $F:\Sigma\times (0,T)\to\rr^{n+1}$ is a $1$-parameter family of immersion.
	\newline
	
   	More precisely, we say that $\Sigma_0\subset\rr^{n+1}$ evolves under an EGF given by the curvature function $\gamma:\overline{\Gamma}\to [0,\infty)$, or $\gamma$-flow for short, if there exists a $1$-parameter family of immersions $F:\Sigma\times[0,T)\to\rr^{n+1}$ such that 
   \begin{align}\label{gamma-flow}
   	\begin{cases}
   		\pI{\dfrac{\partial F}{\partial t}(x,t),\nu(x,t)}=\gamma(\lambda(x,t)),\mbox{ in }\Sigma\times(0,T),
   		\\
   		F(x,0)=F_0(x),
   	\end{cases}
   \end{align}
   where $\lambda(x,t)=(\lambda_1(x,t),\ldots,\lambda_n(x,t))\in\Gamma$, is the principal curvature vector at  $F(x,t)\in\Sigma_t$, $\pI{\cdot,\cdot}$ and $\nu(x,t)$, respectively, are the euclidean inner product and the positively oriented (see Sec.\ref{Sec:Preliminaries}) unit normal vector of $\Sigma_t$ in $\rr^{n+1}$, respectively. 
   \newline
   
   Furthermore, we consider the class of $1$-homogeneous curvature function $$\gamma:\overline{\Gamma}\to[0,\infty),$$
where $\Gamma_+=\set{\lambda\in\rr^n:\lambda_i>0}\subset\Gamma\subset\rr^n$ is a symmetric open cone  and the following properties hold: 
\begin{enumerate}[a)]
\item\label{a)} $\gamma:\Gamma\to(0,\infty)$ is a smooth symmetric, i.e., $\gamma((\lambda_{\sigma(i)}))=\gamma((\lambda_i))$ for every permutation $\sigma$, and positive function. 
\item\label{b)} $\gamma$ is positively $1$-homogeneous, i.e., for every $c>0$, $\gamma(c\lambda)=c\gamma(\lambda)$. 
\item\label{c)} $\gamma$ is increasing on each argument, i.e., $\frac{\partial \gamma}{\partial \lambda_i}(\lambda)>0$.  
\item\label{d} $\gamma$ posses a continuous extension to $0$ at the boundary of $\Gamma$, i.e: there exist a continuous function $\tilde{\gamma}:\overline{\Gamma}\to[0,\infty)$ such that $\restri{\tilde{\gamma}}{\Gamma}=\gamma$ and $\restri{\tilde{\gamma}}{\partial\Gamma}=0$.
\end{enumerate}

To illustrate, in one dimension we only have the curve shortening flow (CSF) in this class of EGFs, where $\gamma(\lambda)=\lambda$. A relevant result  of  CSF is that  simple closed initial curves collapse  into a round point in finite time, i.e., once the curve becomes convex, it deforms into a shrinking circle before collapsing. This phenomenon is known as the development of a singularity under the flow, which is commonly studied with the blow-up rate of a certain extrinsic curvature term of $\Sigma_t$.
  \newline
  
  In higher dimensions, there is no unique analog to CSF due to the presence of multiple principal direction  which a hypersurface in $\rr^{n+1}$ is curved. For this reason, we mention some examples of EGFs that have been of interest due to their diverse applications in geometry and topology: The mean curvature flow (MCF), where the $ \gamma=H=\lambda_1+\dots+\lambda_n$, is the gradient flow of the area functional and it is also an analog of the heat equation for immersions.   The $\alpha$-Gauss curvature flow  ($\alpha$-GCF), where $\gamma=	K^\alpha=(\lambda_1\ldots\lambda_n)^\alpha$ with $\alpha\in\rr$, has been used to deform convex hypersurfaces into ones that attains an isopermetric inequalities such  as Minkowski type of problems. 
  \newline

The class of $\gamma$-flows considered in this article can be viewed as fully nonlinear analogs of the MCF, as the curvature functions are $1$-homogeneous and belong to the class of contractive flows. More specifically, a flow is called contractive when strictly convex (or convex, in the case of the MCF) closed (compact without boundary) initial data collapse to a round point in finite time. Unlike the CSF, in higher dimensions, a normalization of the $\gamma$-flow possesses a subfamily that contracts as spheres toward its center. For a comprehensive introduction to the EGFs, we refer interested readers to the book \cite{chow2020extrinsic}.
  \newline
  
  In this paper, we are interested in a particular type of evolution given by translations in a fixed unit direction, where the solutions are known as translating solitons\footnote{A term from fluid physics, referring to single waves that maintain their shape and velocity over long propagation through a medium.}, or $\gamma$-translators for short. These are self-similar solutions to Eq. \eqref{gamma-flow}, given by
   \begin{align*}
  	F(x,t)=F_0(x)+e_{n+1}t, \mbox{ for all }t\in\rr,
  \end{align*}
  up to tangential diffeomorphisms of $\Sigma_0$, and $e_{n+1}=(0,\ldots,0,1)\in\mathbb{R}^{n+1}$.  
  \newline
  
The interests in $\gamma$-translators for a curvature function $\gamma : \overline{\Gamma} \to [0, \infty)$, as found in the literature, lies in the following topics: From the perspective of EDP, a $\gamma$-translator can furthermore be studied by elliptic PDE theory, since each time slice satisfies a curvature constrain in terms of an angle function of $\Sigma_t$, i.e.,
	\begin{align}\label{gamma-trans}
		\gamma(\lambda) = \langle \nu, e_{n+1} \rangle.
	\end{align}
	
	More recently, as demonstrated in \cite{lynch2023differential}, for this class of $1$-homogeneous curvature functions, $\gamma$-translators serve as models for Type II singularities that may arise under the evolution of their respective $\gamma$-flow,
	as in the case of $2$-convex closed initial data, see Fig. \ref{Dumbbell}. We refer the reader to \cite{cogo2023rotational} for the classification of this singularities.
	\begin{figure}
		\begin{center}
			\includegraphics[width=.4\textwidth, height=2.5cm]{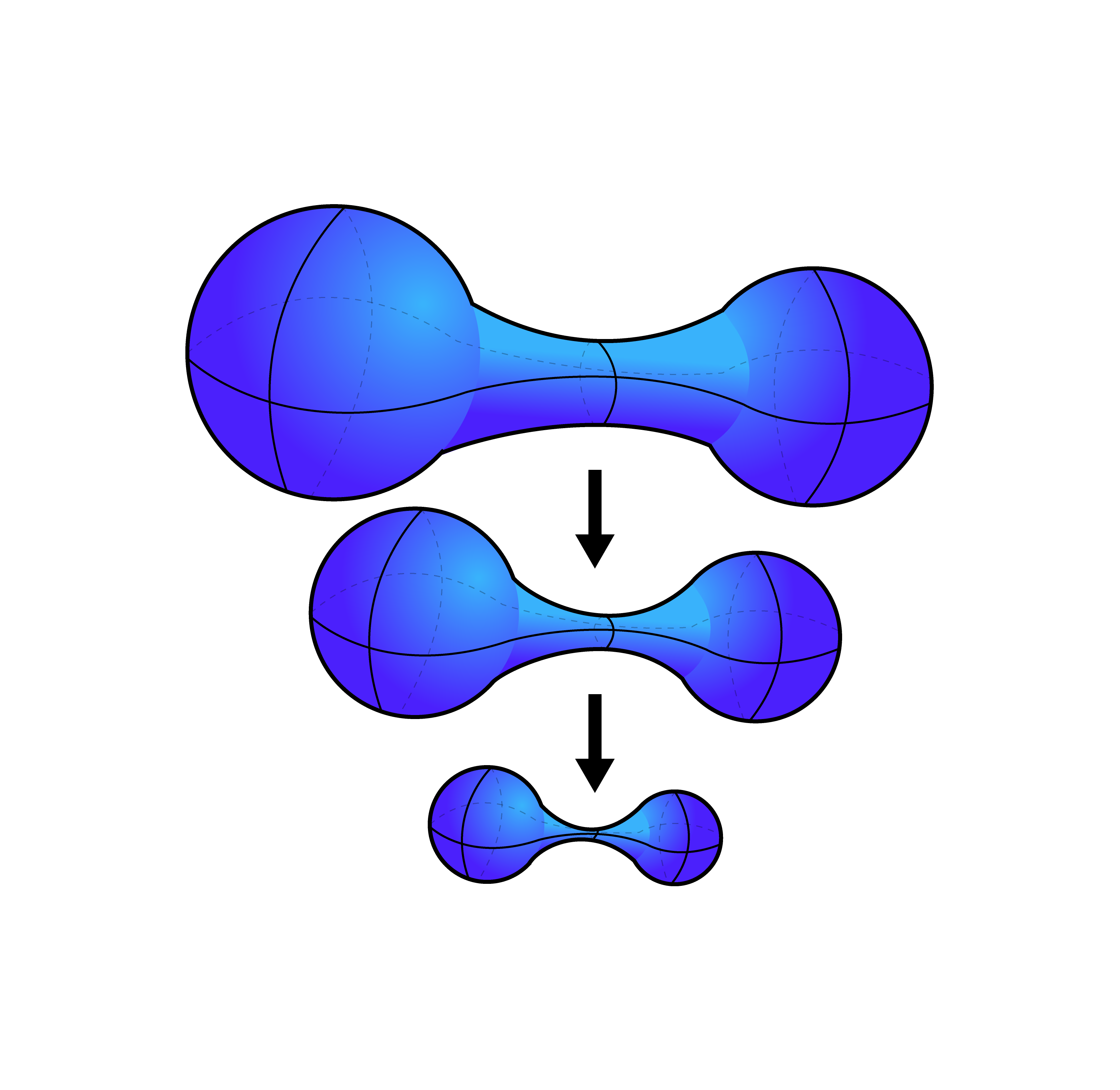}
			\caption{The evolution of a dumbbell under the MCF presenting a type II singularity in its neck.}
			\label{Dumbbell}
		\end{center}
	\end{figure}
  Additionally, $\gamma$-translators serve as candidates for limits of non-compact ancient solutions, i.e., $\Sigma_t$ is an ancient solution to Eq. \eqref{gamma-flow} if $t \in (-\infty, a)$ for some $a \in \mathbb{R}$, and the limit it is considered as $a \to \infty$.
These topics have been studied more extensively for $\gamma=H$ in $\rr^{n+1}$. We refer the reader to the survey \cite{hoffman2017notes} for more details. For recent advances where $\gamma$ is a $\alpha$-homogeneous curvature functions, we suggest the following references: \cite{rengaswami2023ancient} for the construction of ancient solution, so-called ancient pancake that are asymptotic to Grim-Reaper cylinders (see blow); \cite{Shati-Yo} for the asymptotic behavior of bowl-type solutions and their applications; and \cite{Yo} for a uniqueness result of graph solutions smoothly asymptotic to round cylinders. For powers of Gaussian curvature, see \cite{choi2020uniqueness}, and for powers of mean curvature, see \cite{lou2023translating}.
	\newline
	
The main result of this paper generalizes the results of Spruck and Xiao in \cite[Theorem 1.1]{spruck2020complete} and \cite[Theorem 1.1]{xie2023convexity} to nonlinear curvature functions.
	
	\begin{theorem}\label{Generalization of SX}
	Let  $\gamma:\overline{\Gamma}\to[0,\infty)$ be a curvature function and $n\geq 2$. Let $\Sigma\subset\rr^{n+1}$ be a complete $\gamma$-translator. Assume further:
		\begin{enumerate}
			\item  If $n=2$:
			
			\begin{enumerate}
				\item The principal curvatures of $\Sigma$ belong to 
				\begin{align*}
					\Gamma_{\alpha}=\set{\lambda\in\Gamma: \alpha|A|\leq H},\mbox{ with } \alpha\in (0,1).
				\end{align*}
				\item Either  $\gamma(\lambda)$ is concave and $(0,1)\in\Gamma$, or $\gamma(\lambda)$ is convex.
			\end{enumerate}
		\item If $n\geq 3$:
		\begin{enumerate}
		\item The principal curvatures of $\Sigma$ belongs to $$\tilde{\Gamma}= \set{\lambda\in\Gamma: H>0,\mbox{ and }\lambda_1+\lambda_2>0}.$$
			\item  Either  $H\leq \gamma$ and it is strictly convex in off-radial directions, or $\gamma\leq H$ and it  is strictly concave in off-radial directions (see Def. \ref{def}).
 		\end{enumerate}	
		\end{enumerate}
		Then, $\Sigma$ is convex. 
	\end{theorem}
 \begin{remark}
The families of cones $\Gamma_\alpha$ are preserved under $\gamma$-flow. Moreover,  we observe that the points $p\in\Sigma$ such that the principal curvatures $\lambda\in\Gamma_{\alpha}$ with $\alpha\geq1$, are points where the principal curvatures are nonnegative since $\lambda_1+\lambda_2>0,$ and $\lambda_1\lambda_2\geq0$. The condition $(0,1)\in\Gamma$ for concave curvature functions includes the family of $\gamma_t(\lambda)=tH+(1-t)\gamma(\lambda)$, where $t\in (0,t_0]$ for some $t_0\in (0,1]$, and $\gamma(\lambda)$ is a concave curvature function.  Furthermore, this condition is sharp in the sense that if $\gamma(0,1)=0$, then $\Gamma=\Gamma_+$.   Example of strictly concave in off-radial directions curvature functions include the family of Hessian  functions  $\sqrt[k]{S_{k}}$ supported in  	$\Gamma_k=\set{\lambda\in\rr^n: S_i(\lambda)>0,\mbox{ for }i=1,\ldots,k}$,	 where $S_k(\lambda)$ denotes the symmetric elemental polynomial of order $k$ in $n$-the variables, i.e., $S_0=1$, $S_{i}=0$ if $i>n$ and $S_k(\lambda)=\sum\limits_{i\leq i_{1}<\ldots<i_k\leq n}\lambda_{i_{1}}\ldots\lambda_{i_{k}}$.

\end{remark}
	
	In the remainder of the paper, we focus on characterizations of the family of Grim Reaper cylinders in $\rr^{n+1}$ for curvature functions satisfying the non-degeneracy condition  $\gamma(1,0,\ldots,0)=1$, using maximum principle arguments. 
\newline

More precisely, a Grim Reaper cylinder is a graphical hypersurface that is isometric to a Cartesian product of a Euclidean factor $\rr^{n-1}$ with the planar Grim Reaper curve $\left(x,-\ln(\sin(x))\right)$. Additionally, this curve is the unique solution (up to vertical translations) to the one-dimensional version of Eq. \eqref{gamma-trans}.

	\begin{remark}
We note that without the normalization condition $\gamma(1,0,\ldots,0)=1,$ the Grim Reaper cylinders evolve under translations in a non-unitary direction when their associated $\gamma$-flow  is applied.
	\end{remark}
	
The non-existence result obtained in this article is an application of the tangential principle presented in \cite{Yo}, applied to the Grim Reaper cylinders family.
	\begin{theorem}\label{non-existence}
		Let  $\gamma:\overline{\Gamma}\to[0,\infty)$ be a non-degenerate, normalized curvature function. Then, there does not exist a complete, strictly convex $\gamma$-translator $\Sigma\subset\rr^{n+1}$ with compact boundary contained in any non-vertical cylinder.
	\end{theorem}

The following characterizations obtained in this article in terms of curvature constrains generalize the results presented in \cite[Lemma 5.1]{spruck2020complete}, \cite[Thm. B, Thm. B]{Paco_2014,martinez2022mean}, to this family of non-degenerate curvature functions.
	\begin{theorem}\label{convex}
		Let  $\gamma:\overline{\Gamma}\to[0,\infty)$ be a convex, non-degenerate, normalized curvature function.
		Let $\Sigma\subset\rr^{n+1}$ be a complete, convex, non-totally geodesic $\gamma$-translator such that there exists a point $p_0\in\Sigma$ that satisfies  $\lambda_{2}(p_0)=\ldots=\lambda_{n}(p_0)=0$, where the principal curvatures of $\Sigma$ are ordered by $\lambda_n\leq\ldots\leq\lambda_1$. 
		
		Then, $\Sigma$ is a Grim Reaper cylinder. 
		In particular, if $\Sigma$ contains a straight line, then $\Sigma$ is a Grim Reaper cylinder. 
	\end{theorem}
	
	\begin{theorem}\label{T1}
		Let $\gamma:\overline{\Gamma}\to[0,\infty)$ be a concave, non-degenerate, normalize curvature function. Let $\Sigma\subset\rr^{n+1}$ be a complete, convex, non-totally geodesic $\gamma$-translator. Then, $\Sigma$ is a Grim Reaper cylinder if, and only if, the function $|A|^2\gamma^{-2}$ attains a local maximum on $\Sigma$.
	\end{theorem}
	
Finally, as a direct application of Thm.\ref{Generalization of SX} and the above characterizations of the Grim Reaper cylinders, we obtain the following result.
	
	\begin{corollary}\label{Coro}
		Let $\gamma:\overline{\Gamma}\to [0,\infty)$ be a normalized non-degenerate (strictly convex in off-radial directions, for $n>2$) convex/concave  curvature function, and let $\Sigma\subset\rr^{n+1}$ be a complete not totally geodesic  $\gamma$-translator as in Theorem \ref{Generalization of SX}.  Then, $\Sigma$ is a Grim Reaper cylinder or it is strictly convex. 
	\end{corollary}
	
This article is organized as follows: Section \ref{Sec:Preliminaries} reviews the key equations. Sections \ref{Sec Generalization of SX} and \ref{Generalization of SX} prove Theorem \ref{Generalization of SX} for $n=2$ and $n>2$, respectively. Section \ref{Grim} covers Theorems \ref{non-existence}-\ref{T1}, and Section \ref{Conclusion} proves Corollary \ref{Coro}.
\newline	
\textbf{Acknowledgments:} I would like to thank Ignacio McManus for his contributions of images to this work.
		\section{\textbf{Preliminaries}}	\label{Sec:Preliminaries}
In this section, we focus on the equations related to the extrinsic geometry of $\gamma$-translators in $\rr^{n+1}$ with $\gamma:\overline{\Gamma}\to[0,\infty)$ is a fixed curvature function.
\newline

Firstly, we note that the principal curvatures of a hypersurface $\Sigma$ in $\rr^{n+1}$ are the eigenvalues of the shape operator 	$\W_p:T_p\Sigma\to T_p\Sigma$ defined by $\W_p(X)=-\nabla_X\nu$, where $T_p\Sigma$ is the tangent space at  $p\in\Sigma$, $\nu$ is the positively oriented unitary normal vector field of $\Sigma$ in $\rr^{n+1}$. With this orientation, the principal curvatures of a convex graph are non-negative.
\newline

Furthermore, in local coordinates, the shape operator is represented by the matrix $A^i_j=g^{ik}A_{kj}$, where
$g^{ij}$ denotes the coefficients of the inverse of the metric tensor $g=(g_{ij})$ of $\Sigma$, and $A=(A_{ij})$ denotes the coefficients of the second fundamental form of $\Sigma$. We note that Einstein's notation of repeated indexes is used, and will be used throughout this paper as well. To elaborate, the second fundamental form corresponds to the bilinear symmetric form related to the Shape operator, i.e: $A(X,Y)=\pI{-\nabla_X\nu, Y}$. In local coordinate the entries are given by  $A_{ij}=\pI{-\nabla_i\nu,\tau_j}$, where $\set{\tau_i}$ denotes a local frame of  $T_p\Sigma$ and $\nabla_{i}\nu=\nabla_{\tau_i}\nu$.
\newline

Next, we describe how the properties defining a curvature function influence the differential operators in the main equations related to $\gamma$-translators in $\mathbb{R}^{n+1}$:
\begin{enumerate}
	\item Property \ref{a)} implies the existence of a smooth function $$\boldsymbol{\gamma}:\set{A\in\mbox{Sym}(n):\lambda(A)\in\Gamma}\to\rr$$ such that $\boldsymbol{\gamma}(A)=\gamma(\lambda)$ when $A=\mbox{diag}(\set{\lambda_1,\ldots,\lambda_n})$. Here $\mbox{Sym}(n)$ denotes the vector space of real symmetric matrices. 
	\item Property \ref{b)} allows us to write for any diagonal matrix $A=\mbox{diag}(\set{\lambda_1,\ldots,\lambda_n})$ that $\dfrac{\partial\boldsymbol{\gamma}}{\partial A_{ab}}(A) = \delta_{ab}\dfrac{\partial\gamma}{\partial\lambda_a}(\lambda)$, where $\delta_{ab}=1$, if $a=b$, and $\delta_{ab}=0$ otherwise,
	\newline
	Furthermore, if  the principal curvature are simple, i.e., $\lambda_1<\ldots<\lambda_n$, then 
	\begin{align*}
		\dfrac{\partial^2\boldsymbol{\gamma}}{\partial A_{ab}\partial A_{cd}}(A)T_{ab}T_{cd}=\dfrac{\partial^2\gamma}{\partial\lambda_a\partial \lambda_b}(\lambda)T_{aa}T_{bb}+2\sum_{a<b}\dfrac{\frac{\partial\gamma}{\partial\lambda_b}(\lambda)-\frac{\partial\gamma}{\partial\lambda_a}(\lambda)}{\lambda_b-\lambda_a}|T_{ab}|^2,
	\end{align*}
	for every symmetric matrix $T_{ab}$. In particular, if $\gamma$ is convex (resp. concave), the above term has a non-negative (resp. non-positive) sign.
\end{enumerate}

	Throughout this paper, we will abuse notation by using the same symbol for smooth symmetric functions evaluated at both $\lambda\in\Gamma$ and $A\in\set{Sym(n):\lambda(A)\in\Gamma}$.

\begin{lemma}\label{equations}
	Let $p\in\Sigma$ be a $\gamma$-translator, and choose an orthonormal frame of principal directions $\set{\tau_i}_{i=1}^n$ of $T_p\Sigma$, then the following equations holds at $p$:
	\begin{align}
		\label{gamma}
		&\Delta_\gamma\gamma+\nabla_{n+1}\gamma+|A|_\gamma^2 \gamma=0,
		\\
		\label{A_ij}
		&\Delta_\gamma A_{ij}+\partial^2\gamma^{ab;cd}\nabla_iA_{ab}\nabla_jA_{cd}+\nabla_{n+1} A_{ij}+|A|_\gamma^2A_{ij}=0,
		\\ \label{H}
		&\Delta_\gamma H+\partial^2\gamma^{ab;cd}\nabla_i A_{ab}\nabla_i A_{cd}+\nabla_{n+1} H+|A|_\gamma^2H=0,
		\\
		&\Delta_\gamma|A|^2+2\lambda_i\partial^2\gamma^{ab;cd}\nabla_iA_{ab}\nabla_i A_{cd}-2\norm{\nabla A}_\gamma^2+\nabla_{n+1}|A|^2+2|A|_\gamma^2|A|^2=0. 
	\end{align}
	where  
	\begin{align*}
		&\Delta_\gamma=\dfrac{\partial \gamma}{\partial A_{ij}}\nabla_i\nabla_j,\:
		\partial^2\gamma^{ab;cd}\nabla_i A_{ab}\nabla_i A_{cd}=	\dfrac{\partial^2\gamma}{\partial A_{ab}\partial A_{cd}}(A)\nabla_i A_{ab}\nabla_iA_{cd},
		\\
		&\:\norm{X}_\gamma^2=\pI{X,X}_\gamma,
		\pI{X,Y}_\gamma=\dfrac{\partial \gamma}{\partial A_{ij}}X^iY^j, 
	\nabla_{n+1}f=\pI{\nabla f, (e_{n+1})^{\top}}, |A|^2_\gamma=\dfrac{\partial\gamma}{\partial\lambda_i}\lambda_i^2,
	\end{align*}
	and $(\cdot)^\top$ is the orthogonal projection onto $T_p\Sigma$. 
\end{lemma}

\begin{proof}
Let  $p\in\Sigma$ be fixed, and let us choose an adapted orthonormal frame of principal directions centered at $p\in\Sigma$, denoted by $\set{\tau_i}\subset T_p\Sigma$, i.e., $A^i_j(p)=A_{ij}(p)=\delta_{ij}\lambda_i(p)$.  Then, we have 
	\begin{align*}
		\nabla_i\gamma&=\nabla_i\pI{\nu,e_{n+1}}=\pI{\nabla_i\nu, \tau_l}\pI{\tau_l,e_{n+1}}=-A_{il}\pI{\tau_l,e_{n+1}},
		\\
		\nabla_j\nabla_i\gamma&=-\nabla_jA_{il}\pI{\tau_l,e_{n+1}}-A_{il}\pI{\nabla_j\tau_l,\nu}\pI{\nu,e_{n+1}}=-\nabla_jA_{il}\pI{\tau_l,e_{n+1}}-	A_{il}A_{lj}\gamma,
		\\
		\Delta_\gamma\gamma&=-\pI{\nabla \gamma,e_{n+1}}-|A|_\gamma^2\gamma.
	\end{align*}
For the remainder terms, we observe that
	\begin{align*}
		\nabla_i\gamma=&\dfrac{\partial\gamma}{\partial A_{ab}}\nabla_iA_{ab},
		\\
		\nabla_j\nabla_i\gamma
		=&\partial^2\gamma^{ab;cd}\nabla_iA_{ab}\nabla_jA_{cd}+\partial\gamma^{ab}\nabla_j\nabla_iA_{ab}
		\\
		=&\partial^2\gamma^{ab;cd}\nabla_iA_{ab}\nabla_jA_{cd}
		\\
		&+\partial\gamma^{ab}\left( \nabla_a\nabla_b A_{ji}+A_{ji}A_{al}A_{bl}-A_{jl}A_{ai}A_{lb}+A_{jb}A_{al}A_{li}-A_{jl}A_{ab}A_{li}\right).
	\end{align*}
	Therefore, we see that
	\begin{align*}
		\Delta_\gamma A_{ij}&=-\partial^2\gamma^{ab;cd}\nabla_iA_{ab}\nabla_jA_{cd}-|A|_\gamma^2A_{ij}+\gamma|A|^2+\nabla_i\nabla_j\gamma
		\\
		&=-\partial^2\gamma^{ab;cd}\nabla_iA_{ab}\nabla_jA_{cd}-|A|_\gamma^2A_{ij}-\nabla_{n+1}A_{ij}.		
	\end{align*}
	Finally, the equations for 
	$H$  and $|A|^2$ follow by taking the trace of the above equation and noting that $\nabla_i\nabla_j|A|^2=2A_{pq}\nabla_{i}\nabla_{j}A_{pq}+2\nabla_i A_{pq}\nabla_j A_{pq}$.
\end{proof}

\begin{proposition}
	Let $\Sigma\subset\mathbb{R}^{n+1}$ be a $\gamma$-translator. Then, at  a point $p\in\Sigma$ where the principal curvatures are simple, the following equations hold:
	\begin{align}\label{Dif lambda}
		\nabla \lambda_i=(\nabla A)(\tau_i,\tau_i)=\nabla A_{ii}\mbox{ and }\Delta_\gamma\lambda_i=\Delta_\gamma A_{ii}-2\sum_{i\neq j}\dfrac{\norm{\nabla A_{ij}}_\gamma^2}{\lambda_j-\lambda_i}, 
	\end{align}  
	where $\set{\tau_i}\subset T_p\Sigma$ is an orthonormal frame of principal directions.
\end{proposition}
\begin{proof}
	The proof is a straightforward adaptation of \cite[Lemma 2.2]{xie2023convexity}. 
\end{proof}

\begin{definition}
Let $\gamma:\overline{\Gamma}\to[0,\infty)$ be a curvature function. We say that $\gamma$ is uniformly eliptic if there exists a positive constant $C=C(n,\gamma,\Gamma)$ such that 
	\begin{align*}
		&C^{-1}\delta_{ab}\dfrac{\partial\gamma}{\partial\lambda_a}(\lambda)\leq \dfrac{\partial\gamma}{\partial{A_{ab}}}(A)\leq C\delta_{ab}\dfrac{\partial\gamma}{\partial\lambda_a}(\lambda).
	\end{align*}
\end{definition}

\begin{lemma}\label{Omori-Yau}
	Let $\gamma:\overline{\Gamma}\to[0,\infty)$ be a uniformly elliptic curvature function. Let $\Sigma$ be a complete $\gamma$-translator with principal curvatures belonging to $\Gamma$ and with $|A| \leq \tilde{C}$ for some constant $\tilde{C} > 0$. Then, for a smooth function $f:\Sigma\to\rr$ that is bounded from below, there exists a sequence $p_m\in\Sigma $ satisfying the Omori-Yau maximum principle related to $\Delta_\gamma$, i.e.,
	\begin{align}\label{Omo}
		f(p_m)\to\inf_{\Sigma}f,\:|\nabla f(p_m)|\to0\mbox{ and }\Delta_\gamma f(p_m)\to\delta\geq 0 \:, \mbox{ as }m\to\infty. 
	\end{align}
	The theorem also holds if we replace the assumption that $f$ is smooth with the assumption that $f$ is smooth in $\{f < a\}$ for some $a>\inf\limits_\Sigma f$.
\end{lemma}
\begin{proof}
	This is a direct consequence of \cite[Corollary 3]{alias2013hypersurfaces}. Indeed, we choose $P:T\Sigma\to T\Sigma$, where in a local moving frame the coefficients are given by $\left(\frac{\partial\gamma}{\partial A^i_j}\right)$, where $A^i_j$ is the local expression of the coefficients of the shape operator. Furthermore, the Ricci condition is satisfied by the Gauss-Codazzi equations. Specifically, we have that the Ricci curvature satisfies $R_{ij}=HA^i_{j}-A^i_{k}A^k_{j}$ and $|A|\leq \tilde{C}$. Finally, the limits in \eqref{Omo} hold since  $\gamma$ is a uniformly elliptic curvature function.
\end{proof}

\begin{lemma}\label{Elliptic estimate n=2}
	Let $\gamma:\overline{\Gamma}\subset\rr^2\to[0,\infty)$ be a curvature function. Then, for every $\alpha<1$, the restriction of $\gamma$ onto $\Gamma_{\alpha}=\set{\lambda\in\Gamma: \alpha |A|\leq H}$ is uniformly elliptic.		
\end{lemma}
\begin{proof}
We refer the reader to \cite[Corollary 2.4]{andrews2015convexity} for a proof of this result.
\end{proof}

\begin{definition}\label{def}
Let $\gamma:\overline{\Gamma}\to[0,\infty)$ be a curvature function. We say that $\gamma(\lambda)$ is strictly convex (resp. concave) in off-radial directions if 
\begin{align*}
\pI{\mbox{Hess}(\gamma)(\lambda)\xi,\xi}\leq 0 \:(\mbox{resp.}\geq 0),
\end{align*} 
for all $\lambda\in\Gamma\mbox{ and }\xi\in\rr^n$, with equality holding if, and only if, $\lambda$ is a scalar multiple of $\xi$. 
\end{definition}

\begin{lemma}\label{Eliptic estimate}
	Let $\gamma:\overline{\Gamma}\to [0,\infty)$ be a strictly convex/concave in off-radial direction curvature function. Let $\Gamma'$ be a symmetric closed cone such that is compactly supported in $\set{\lambda\in\Gamma:\lambda_1<0}$, i.e., $\overline{\Gamma'}\cap\Sp^{n-1}$ is compact in $\set{\lambda\in\Gamma:\lambda_1<0}$. Then, $\gamma$ is uniformly elliptic in $\Gamma'$ and for the second order terms it follows
	\begin{itemize}
		\item $\partial^2\gamma^{ab,cd}(A)T_{iab}T_{icd}\leq -C\dfrac{|T|^2}{\mbox{Tr}(A)}$ when $\gamma$ is strictly concave. 
		
		\item $\partial^2\gamma^{ab,cd}(A)T_{iab}T_{icd}\geq C\dfrac{|T|^2}{\mbox{Tr}(A)}$ when $\gamma$ is strictly convex. 
	\end{itemize}
	where $A$ is a diagonal matrix with eigenvalues $\lambda\in\Gamma'$ and $T_{iab}$ is a totally symmetric tensor.
\end{lemma}

\begin{proof}
We refer the reader to \cite{lynch2020convexity} for a proof of the concave case. The proof for the convex case follows analogously from the concave case.
\end{proof}
		
	\section{\textbf{Proof of Theorem \ref{Generalization of SX} for $n=2$}} \label{Sec Generalization of SX}
	In this section we will prove Thm. \ref{Generalization of SX} for $n=2$. Therefore, we assume that $\Sigma\subset\rr^3$ is  a complete $\gamma$-translator with 	$\lambda\in\Gamma_\alpha=\set{\lambda\in\Gamma: \alpha|A|\leq  H}$ with $\alpha\in (0,1)$. In addition, through this section we label the principal curvatures as $\lambda_1\leq\lambda_2$. 
	\newline
	
	Then, by Lemma \ref{Elliptic estimate n=2}, $\gamma$ is an uniformly elliptic curvature function, i.e., there exists a positive constant $C=C(\alpha,\gamma,\Gamma)$ such that $C^{-1}H\leq \gamma\leq CH$. This implies that $\Sigma$ has uniformy bounded second fundamental form,  since $\gamma=\pI{\nu,e_{n+1}}\leq 1$ and $|A|\leq C\alpha^{-1}\gamma$ in $\Sigma$.
	\newline
	
Moreover,  since $\alpha<1$, we have
\begin{align}\label{c(a)}
-1<c(\alpha)=\alpha^2-1<\dfrac{\lambda_1}{\lambda_2}<0,
\end{align}
 for every  $p\in\Sigma^{-}=\set{p\in\Sigma:\lambda_1(p)<0}$.
\newline	
	
We begin by proving Thm. \ref{Generalization of SX} for the case where $\gamma:\overline{\Gamma}\to[0,\infty)$ is a concave curvature function with $(0,1)\in\Gamma$. For this purpose, we consider the function $h:\Sigma\to(0,\infty)$ given by 
	\begin{align*}
		h(p)&=\dfrac{\gamma(\lambda)}{\lambda_2}=\gamma\left(\frac{\lambda_1}{\lambda_2},1\right).
	\end{align*}
We note that $h(p)$ is smooth in the open set  $\Sigma^{-}$, as it does not contain any umbilical points of $\Sigma$. Moreover, since $\gamma$ is strictly increasing in each argument, $h(p)\in (\gamma(c(\alpha),1),\gamma(1,1)]$.
\newline

Subsequently, we may write
\begin{align*}
	\Sigma^-=\set{p\in\Sigma: h(p)\in\left(\gamma(c(\alpha),1),\gamma(0,1)\right)}.
\end{align*}
Therefore, the convexity of $\Sigma$ will follow by establishing that
\begin{align*}
	I=\inf\limits_{p\in\Sigma} h(p)\geq \gamma(0,1).
\end{align*}
In the following lemmas, we will argue by contradiction that  $I\in (\gamma(c(\alpha),1),\gamma(0,1))$ cannot occur.
	
	\begin{lemma}\label{Step 1}
	The function $h(p)$ cannot attain an interior minimum in $\Sigma^{-}$.
	\end{lemma}
	
	\begin{proof}
		We proceed by contradiction, assuming that there exists a point $p_0\in \Sigma^{-}$ such that  $h(p_0)=I$. Then, by choosing an adapted frame of principal directions $\tau_i$ around a fixed point $p\in\Sigma^{-}$, the equations in Lemma \ref{equations} hold. Therefore, we have the following equations related to $h$:
		\begin{align}\label{gradient h}
			\nabla h=&\dfrac{\nabla \gamma}{\lambda_2}-h\dfrac{\nabla \lambda_2}{\lambda_2},
		\end{align}	
		\begin{align*}	
			\Delta_\gamma h=&\dfrac{\Delta_\gamma \gamma}{\lambda_2}-\pI{\dfrac{\nabla \gamma}{\lambda_2},\dfrac{\nabla\lambda_2}{\lambda_2}}_\gamma-\pI{\nabla h,\dfrac{\nabla\lambda_2}{\lambda_2}}_\gamma+ h\norm{\dfrac{\nabla\lambda_2}{\lambda_2}}_\gamma^2-\dfrac{h}{\lambda_2}\Delta_\gamma\lambda_2
			\\\notag
			=&-2\pI{\nabla h,\dfrac{\nabla\lambda_2}{\lambda_2}}_\gamma-\nabla_3h+\dfrac{h}{\lambda_2}\left(\partial^2\gamma^{ab;cd}\nabla_2A_{ab}\nabla_2A_{cd}-2\dfrac{\norm{\nabla A_{12}}_\gamma^2}{\lambda_2-\lambda_1}\right).
		\end{align*}
	
		Then,  since at $p\in \Sigma^-$ the principal curvatures are simple, Property \ref{c)} implies 
		\begin{align*}
			\norm{\nabla A_{12}}^2_\gamma=(\nabla_1A_{12},\nabla_2 A_{12})\begin{pmatrix}
			\dfrac{\partial\gamma}{\partial A_{11}}& \dfrac{\partial\gamma}{\partial A_{12}}
			\\
			\dfrac{\partial\gamma}{\partial A_{21}}& \dfrac{\partial\gamma}{\partial A_{22}}
			\end{pmatrix}
		\begin{pmatrix}
				\nabla_ 1 A_{12}
				\\
				\nabla_2 A_{12}
			\end{pmatrix} \geq0.
		\end{align*}
Moreover, since $\gamma$ is a concave curvature function, we observe that
		\begin{align*}
			\partial^2\gamma^{ab;cd}\nabla_2A_{ab}\nabla_2A_{cd}=\dfrac{\partial^2\gamma}{\partial\lambda_a\partial\lambda_b}\nabla_2 A_{aa}\nabla_2 A_{bb}+2\sum_{a<b}\dfrac{\frac{\partial\gamma}{\partial\lambda_b}-\frac{\partial\gamma}{\partial\lambda_a}}{\lambda_b-\lambda_a}|\nabla_2 A_{ab}|^2\leq 0.
		\end{align*}
		
		 Therefore, we obtain the following differential inequality
		\begin{align}\label{h}
		\Delta_\gamma h+2\pI{\nabla h,\dfrac{\nabla \lambda_2}{\lambda_2}}_\gamma+\nabla_3 h=\dfrac{h}{\lambda_2}\left(\partial^2\gamma^{ab;cd}\nabla_2A_{ab}\nabla_2A_{cd}-2\dfrac{\norm{\nabla A_{12}}_\gamma^2}{\lambda_2-\lambda_1}\right)\leq 0. 
		\end{align}
		Consequently, since $h(p)$  reaches its minimum at $p_0\in\Sigma^{-}$, by the strong maximum principle, we conclude that $h$ is constant on $\Sigma^{-}$. 
		\newline
		
		Now, since $h$ is constant on $\Sigma^{-}$,  from Eq. \eqref{h}, we obtain 
			\begin{align*}
				0=\norm{\nabla A_{12}}_\gamma^2\mbox{ in }\Sigma^-.
			\end{align*}
		This implies that $0=\nabla_1 A_{12}=\nabla_2 A_{12}$, and by the Gauss-Codazzi equations, we observe that  $\nabla_{2}A_{11}=\nabla_1A_{22}=0$ in  $\Sigma^-$. In other words,  $\set{\tau_i}$ is a parallel orthonormal frame of $\Sigma^-$, and therefore, $\lambda_1\lambda_2=0$ in $\Sigma^-$, gives a contradiction since $\lambda_1(p_0)<0<\lambda_2(p_0)$.
	\end{proof}

	Next, we consider the case where the infimum is attained at infinity in $\Sigma^-$, i.e., there is an unbounded sequence  $p_k \in \Sigma^-$ such that $\lim\limits_{k\to\infty}h(p_k) = I$. To address this, we apply Lemma \ref{Omori-Yau} on $\Sigma^-$ (possibly after considering a subsequence) to obtain a sequence $p_k \in \Sigma^-$ such that
	\begin{align}\label{Omori} 
		h(p_k) \to I, \:|\nabla h(p_k)| \to 0, \text{ and } \Delta_\gamma h(p_k) \to \delta \geq 0 \text{ as } k \to \infty. 
	\end{align}

	Then, the sequence of surfaces
	\begin{align*}
		\Sigma_k=\Sigma-p_k=\set{p-p_k\in\rr^3:p\in\Sigma}
	\end{align*}
	satisfies the following: $\Sigma_k$ is  a $\gamma$-translator, $0\in\Sigma_k$, and possess uniformly bounded second fundamental form for every $k$. Therefore, by compactness, the sequence $\Sigma_k$ possesses a subsequence that converge smoothly on compact subset to a surface $\Sigma_\infty$, which is not necessarily connected. However, we note that smooth convergence on compact sets implies that $\Sigma_\infty$ is a complete $\gamma$-translator.
	\newline
	
	We claim that  the connected component $\Sigma_\infty'$  of $\Sigma_\infty$ that contains the origin is a vertical plane. Indeed, if this were not the case, then $\gamma>0$ on $\Sigma_\infty'$, and consequently, the function $h(p)$ attains a positive infimum at the origin of $\Sigma_\infty'$ contradicting Lemma \ref{Step 1}. 
	
	\begin{lemma}The infimum $I$ cannot be attained at infinity on $\Sigma$.
	\end{lemma}
	\begin{proof}
		Firstly, we note that the quadratic form given by
		\begin{align}\label{quadratic}
			\dfrac{1}{\lambda_2(p_k)}A(\cdot,\cdot)(p_k)
		\end{align}
 will converge to a quadratic form in $\Sigma_\infty'$ with eigenvalues $1$ and $a<0$. To see this, we first note that the $1$-homogeneity of $\gamma$ implies that
		\begin{align*}
			\gamma(\lambda)=\partial_1\gamma(\lambda)\lambda_1+\partial_2\gamma(\lambda)\lambda_2.
		\end{align*}
		  Then, for every $p_k\in \Sigma^{-}$, we observe that
		  \begin{align*}
		  -\dfrac{\lambda_1}{\lambda_2}&=\dfrac{\partial_2\gamma}{\partial_1\gamma}-\dfrac{\partial_2\gamma\lambda_2+\partial_1\gamma\lambda_1}{\partial_1\gamma\lambda_2}= \dfrac{\partial_2\gamma-h}{\partial_1\gamma}\geq C\left(\partial_2\gamma(\lambda_1,1)-h\right),
	\end{align*}
in the last inequality we used that $\gamma$ is uniformly elliptic and $\partial_2\gamma(\lambda_1,\lambda_2)$ is decreasing in the second argument do to the concavity of $\gamma(\lambda)$. Consequently, the claim follows since 
\begin{align*}
a=\lim_{k\to\infty}\dfrac{\lambda_1}{\lambda_2}
\leq C(I-\partial_2\gamma(0,1))=C(I-\gamma(0,1))<0,	
\end{align*}
where  $\lim\limits_{k\to\infty}h(p_k)=I<\gamma(0,1)$, $\lim\limits_{k\to\infty}\lambda_i(p_k)=0$ and $\gamma(0,1)=\partial_2\gamma(0,1)$. 
\newline		
		
		On the other hand, by taking derivatives of $\gamma$ at $p_k$, we observe that 
		\begin{align}\label{nabla gamma}
			\nabla \gamma=\nabla \pI{\nu,e_3}=A(e_3^\top,\cdot),
		\end{align}
		where $(\cdot)^{\top}$ denotes the orthogonal projection onto $T\Sigma$. At $p_k$, we have that  $e_3^\top\to e_3$ as $k\to\infty$ since $\Sigma_\infty'$ is a vertical plane passing through the origin. This gives 
		\begin{align*}
			\dfrac{\nabla \gamma}{\lambda_2}(p_k)\to \vec{v}\in\rr^3\setminus\set{ 0}, \mbox{ as }k\to\infty,
		\end{align*}
		because the eigenvalues of the quadratic form \eqref{quadratic} are non-zero. Consequently, from equations \eqref{gradient h} and  \eqref{Omori}, we obtain 
		\begin{align*}
		 h(p_k)\dfrac{\nabla\lambda_2(p_k)}{\lambda_2(p_k)}\to \vec{v}, \mbox{ as }k\to\infty.
		\end{align*}
	
		Furthermore, by writing $\nabla h$ differently, we see that 
		\begin{align*}
			\nabla h= \dfrac{\partial_1\gamma\nabla\lambda_1+\partial_2\gamma\nabla\lambda_2 }{\lambda_2}- h\dfrac{\nabla\lambda_2}{\lambda_2}.
		\end{align*}
		Then, by multiplying the above equation by $h$, we have 
		\begin{align*}
			h\nabla h
			&=\left(\partial_2\gamma- h \right)h\dfrac{\nabla\lambda_2}{\lambda_2}+h\partial_1\gamma\dfrac{\nabla\lambda_1}{\lambda_2}.
		\end{align*}
		Consequently, by evaluating at $p_k$ and taking limits, we obtain  
		\begin{align*}
		\partial_1\gamma h\dfrac{\nabla\lambda_1}{\lambda_2}&\to
			-\left(\partial_2\gamma(a,1)-I\right)\vec{v},\:\partial_2\gamma h\dfrac{\nabla\lambda_2}{\lambda_2}&\to
		\partial_2\gamma(a,1)\vec{v}.
	\end{align*}

		On the other hand, by multiplying  Eq. \eqref{h} by $h$, we see that 
		\begin{align}\label{hh}
			h\Delta_\gamma h+2 h\pI{\nabla h,\dfrac{\nabla \lambda_2}{\lambda_2}}_\gamma+h\nabla_3 h\leq -2h^2\dfrac{\norm{\nabla A_{12}}_\gamma^2}{\lambda_2(\lambda_2-\lambda_1)}.
		\end{align}
Moreover, we can write the right-hand side of the above equation as
		\begin{align*}
			\dfrac{h^2\norm{\nabla A_{12}}_\gamma^2}{\lambda_2(\lambda_2-\lambda_1)}
			&=\dfrac{1}{1-\dfrac{\lambda_1}{\lambda_2}}\left( \partial_1\gamma(\lambda)\pare{\dfrac{h|\nabla_2\lambda_1|}{\lambda_2}}^2+\partial_2\gamma(\lambda)\pare{\dfrac{h|\nabla_1\lambda_2|}{\lambda_2}}^2 \right)
			\\
			&\to
				\dfrac{1}{1-a}\left(\dfrac{(\partial_2\gamma(a,1)-I)^2|v_2|^2}{\partial_2\gamma(a,1)}+\dfrac{|v_1|^2}{\partial_2\gamma(a,1)}\right)>0,
				\end{align*}
 as $k\to\infty$, where $v_i=\pI{\vec{v},\tau_i}$ and  $\tau_i=\lim\limits_{k\to\infty}\tau_i(p_k)$ is an orthonormal basis of  $T_0\Sigma_\infty'$.  
 \newline
 
 Consequently, by taking limits in Eq. \ref{hh} at $k\to\infty$, we obtain $0\leq \delta I<0$, and therefore, that infimum of $h$ cannot be attained at infinity,  finalizing the proof.
	\end{proof}
	
Now, we proceed with the case where $\gamma:\overline{\Gamma}\to[0,\infty)$ is a convex curvature function. The proof will follow the same steps as in the concave case but will use a different function. Therefore, we will only highlight the main differences from the previous case.
\newline

We consider the function  $j:\Sigma\to\rr$ given by $	j(p)=\dfrac{\lambda_1}{\gamma(\lambda)}$.  We note that $j(p)$ is continuous and bounded, since 
 \begin{align*}
\dfrac{c(\alpha)}{\gamma(1,1)}\leq \dfrac{\frac{\lambda_1}{\lambda_2}}{\gamma\left(\frac{\lambda_1}{\lambda_2},1\right)}= \dfrac{\lambda_1}{\gamma(\lambda)}\leq \dfrac{C\lambda_1}{H}\leq C.
 \end{align*}
Moreover, as in the previous case, $j(p)$ is smooth in $\Sigma^-$. 
\newline

Consequently, $\Sigma$ will be convex if we show that  $J=\inf\limits_\Sigma j(p)\geq 0$. We will achieve this by assuming the contrary, i.e., 		$J\in\left(\dfrac{c(\alpha)}{\gamma(1,1)},0\right)$.
	\begin{lemma}\label{non minimum}
		The function $j(p)$ cannot attain a local  negative minimum in $\Sigma^-$. 	
	\end{lemma}
	\begin{proof}
		Let us assume that there exists $p_0\in\Sigma$ such that  $j(p_0)=J$. Then, in  an  adapted frame of principal directions at $p\in\Sigma$, the following equations hold:
		\begin{align}\label{gradient j}
			\nabla j=&\dfrac{\nabla\lambda_1}{\gamma}-j\dfrac{\nabla \gamma}{\gamma},
		\end{align}
\begin{align*}		
			\Delta_\gamma j=&\dfrac{\Delta_\gamma \lambda_1}{\gamma}-\pI{\dfrac{\nabla \lambda_1}{\gamma^{\frac1\alpha}},\dfrac{\nabla\gamma}{\gamma}}_\gamma-\pI{\nabla j,\dfrac{\nabla\gamma}{\gamma}}_\gamma+ j\norm{\dfrac{\nabla\gamma}{\gamma}}_\gamma^2-j\dfrac{\Delta_\gamma\gamma}{\gamma}
		\\
		=&-2\pI{\nabla j,\dfrac{\nabla\gamma}{\gamma}}_\gamma-\nabla_3j
		-\dfrac{1}{\gamma}\left(\partial^2\gamma^{ab;cd}\nabla_1A_{ab}\nabla_1A_{cd}+2\dfrac{\norm{\nabla A_{12}}_\gamma^2}{\lambda_2-\lambda_1}\right).
			\end{align*}
		Moreover, we observe that 
			\begin{align}\label{j}
			\Delta_\gamma j+2\pI{\nabla j,\dfrac{\nabla\gamma}{\gamma}}_\gamma+\nabla_3j=-\dfrac{1}{\gamma}\left(\dfrac{\partial^2\gamma}{\partial A_{ab}\partial A_{cd}}\nabla_1A_{ab}\nabla_1A_{cd}+2\dfrac{\norm{\nabla A_{12}}_\gamma^2}{\lambda_2-\lambda_1}\right).
		\end{align}
		Therefore, by the convexity of $\gamma$, we have that the right-hand side of Eq. \eqref{j} is nonpositive. Finally, we note that from this point onward, Lemma \ref{non minimum} is satisfied by arguments analogous to those used at the end of Lemma \ref{Step 1}.
	\end{proof}
	
	As in the concave case, by the Omori-Yau maximum principle, we may choose a sequence $p_k\in\Sigma^{-}$ such that 
	\begin{align}\label{Omoris}
		j(p_k)\to J<0,\:|\nabla j(p_k)|\to0\mbox{ and }\Delta_\gamma j(p_k)\to\epsilon\in[0,\infty), \mbox{ as }k\to\infty. 
	\end{align}
	Then, the sequence of $\gamma$-translators $\Sigma_k=\Sigma-p_k$ possesses a subsequence that converges smoothly on compact sets to a surface $\Sigma_\infty$. In addition, as in the concave case, we have that the connected component  $\Sigma_\infty'$ of $\Sigma_\infty$ that contains the origin is a vertical plane.
	\begin{lemma}
	The function $j(p)$ cannot attain a negative infimum at infinity in $\Sigma^-$.
	\end{lemma}
	\begin{proof}
		Firstly, we observe that the quadratic form $\gamma^{-1}(p_k)A(\cdot,\cdot)$
	 will converge to a quadratic form in $\Sigma_\infty'$ with eigenvalues $J<0$ and $b\in\left[-J, \dfrac{1}{\gamma(c(\alpha,1))}\right]$, since 
	 \begin{align*}
	b=\lim_{k\to\infty}\dfrac{\lambda_2(p_k)}{\gamma(p_k)}\geq \lim_{k\to\infty}-\dfrac{\lambda_1(p_k)}{\gamma(p_k)}\geq - J>0. 
	 \end{align*}
	 
	In addition, by Eq. \eqref{nabla gamma}, along with the fact that $\Sigma_\infty'$ is a vertical hyperplane passing through the origin, we have
		\begin{align*}
			\dfrac{\nabla \gamma}{\gamma}(p_k)\to \vec{w}\in\rr^3\setminus\set{ 0}, \mbox{ as }k\to\infty. 
		\end{align*}
Then, by combining equations \eqref{gradient j} with \eqref{Omoris} 
		we obtain  $\dfrac{\nabla\lambda_1}{\gamma}(p_k)\to J\vec{w}$ as $k\to\infty$. Furthermore, by  writing $\nabla j$ differently, we get
		\begin{align*}
			\nabla j&=\dfrac{\nabla \lambda_1}{\gamma}-\dfrac{j}{\gamma} (\partial_1\gamma\nabla\lambda_1+\partial_2\gamma\nabla\lambda_2 )
			=\left(1-j\partial_1\gamma \right)\dfrac{\nabla\lambda_1}{\gamma}-j\partial_2\gamma\dfrac{\nabla\lambda_2}{\gamma}.
		\end{align*}

In particular, after taking limit at $p_k$ as $k\to\infty$, we obtain
\begin{align*}
	j\partial_2\gamma\dfrac{\nabla \lambda_2}{\gamma}\to \left(1-J\partial_1\gamma\left(\dfrac{J}{b},1\right)\right)J\vec{w}. 
\end{align*}
Here, we used that 	$\dfrac{\lambda_1(p_k)}{\lambda_2(p_k)}\to \dfrac{J}{b}<0$ and $\partial_1\gamma$ is a $0$-homogeneous function. 
\newline

Then, after multiply Eq. \eqref{j} by $j(p_k)<0$, we have 
		\begin{align}\label{jj}
			j\Delta_\gamma j+2j\pI{\nabla j,\dfrac{\nabla \gamma}{\gamma}}_\gamma+j\nabla_3 j\geq-\dfrac{2j}{\gamma}\dfrac{\norm{\nabla A_{12}}_{\gamma}^2}{\lambda_2-\lambda_1}.
		\end{align}
	Additionally, we note that at  $p_k$,  the term
	\begin{align*}
		\dfrac{2j}{\gamma}\dfrac{\norm{\nabla A_{12}}_{\gamma}^2}{\lambda_2-\lambda_1}&
		=\dfrac{-2}{1-\frac{\lambda_2}{\lambda_1}}\left(\partial_1\gamma\dfrac{|\nabla_2\lambda_1|^2}{\gamma^{2}}+\partial_2\gamma\dfrac{|\nabla_1\lambda_2|^2}{\gamma^{2}}\right)		
		\\
		&\to \dfrac{-2J^2}{1-\frac{b}{J}}\pare{\partial_1\gamma(b^{-1}J,1)|w_2|^2  +\dfrac{(1-J\partial_1\gamma(b^{-1}J,1))}{\partial_2\gamma(b^{-1}J,1)}|w_1|^2}<0,
	\end{align*}
	 as $k\to\infty$. Here, $w_i=\pI{\vec{w},\tau_i}$ and $\tau_i$ is an orthonormal basis of  $T_0\Sigma_\infty'$. Therefore, by evaluating Eq. \eqref{jj}  at $p_k$ and using Eq. \eqref{Omori} to takes limits, we obtain the desire contradiction $0< J\epsilon\leq 0$, finalizing the proof of the Thm \ref{Generalization of SX} for $n=2$.
	\end{proof}
	
	\section{\textbf{Proof of Theorem \ref{Generalization of SX} for $n\geq 3$}}\label{n>2}
	In this section we will prove  Thm. \ref{Generalization of SX} for dimension $n\geq 3$. In this case, we assume that $\Sigma\subset\rr^{n+1}$ is a complete $\gamma$-translator such that the principal curvatures are labeled by $\lambda_1\leq\ldots\leq\lambda_n$ and belong to  
	\begin{align*}
	\tilde{\Gamma}= \set{\lambda\in\Gamma: H>0,\mbox{ and }\lambda_1+\lambda_2>0}.
	\end{align*} 

Firstly, since $n\geq 3$, we have that $S_2(\lambda)>0$, and therefore 
\begin{align*}
|A|^2=H^2-2S_2(\lambda)\leq H^2. 
\end{align*}
Furthermore, we note that $\tilde{\Gamma}$ is compactly supported in $\Gamma$, i.e., $\overline{\tilde{\Gamma}}\cap\Sp^{n-1}$ is compact in $\Gamma\cap\set{\lambda_1<0}$. By Lemma \ref{Eliptic estimate},  there exists  $C=C(n,\tilde{\Gamma}, \gamma)>0$ such that $C^{-1}\leq \dfrac{\partial\gamma}{\partial\lambda_i}\leq C$, for all $i=1,\ldots,n$. Consequently, since $\gamma$ is $1$-homogeneous, we have $\gamma H^{-1}\in[C^{-1},C]$ in $\lambda\in\tilde{\Gamma}\cap\set{\lambda_1<0}$. 
\newline

Next, throughout the proof of Theorem \ref{Generalization of SX}, we will adopt the notation established in \cite{xie2023convexity}. This will allow us to extend the techniques from the previous section to higher dimensions effectively.
	\newline
	
	 Let $\set{E_i}_{i=1}^{m}$ be the eigendistributions associated with the second fundamental form $A(X,Y)=-\pI{\nabla_X\nu,Y}$,where $\nu$ is the positive  unit normal vector field of $\Sigma$. Then, the set
	\begin{align*}
		\Sigma_A^{-}=\set{p\in \Sigma^{-}: E_i(p)  \mbox{ is constant  on a neighborhood of }  p}	
	\end{align*}
	is an open dense subset  of $\Sigma^-=\set{p\in\Sigma:\lambda_1(p)<0}$. 
	\newline
	
	Furthermore, for every $p\in\Sigma_A^{-} $, we choose the associated orthonormal basis of  principal directions
	\begin{align*}
		\tau_1,\ldots,\tau_{d_1}\in E_1,\:\tau_{d_1+1},\ldots,\tau_{d_2}\in E_2,\ldots,\tau_{d_1+d_2+\ldots d_{m-1}},\ldots,\tau_n\in E_m,
	\end{align*}
	associated with  $\lambda_i=A(\tau_i,\tau_i)$, counted without multiplicity.  
	\newline
		
Now, we begin with the case where $\gamma$ is a strictly convex curvature function in off-radial directions, such that $H \leq \gamma$. For this purpose, we consider the function $g:\Sigma\to (-\infty,0)$ given by 
	\begin{align}\label{f}
		g=f\left(\dfrac{\lambda_1}{\gamma(\lambda)}\right),
\mbox{ where }
	f(r)=\begin{cases}
		-r^4e^{\frac{-1}{r^2}}&,\mbox{if }r<0,
		\\
		0&,\mbox{if }r\geq 0.
	\end{cases}
\end{align}	
This function satisfies $f''<0$ and $f'>0$ where
both  are bounded away from $0$ on $[-r_2, -r_1]$, for any $0<r_1 < r_2$. 
\newline

Moreover, we note that $g(p)$ is bounded in $\Sigma$, since
\begin{align*}
\dfrac{-1}{C(n-2)}	\leq \dfrac{-\lambda_1}{CH}\leq\dfrac{|\lambda_1|}{\gamma} <\dfrac{\lambda_2}{\gamma}\leq\ldots\leq\dfrac{\lambda_n}{\gamma}\leq\dfrac{H}{\gamma}\leq 1,
\end{align*}
and it is smooth on $\Sigma^-$. Then, as in the previous section, we have the following equations for $g$ in $\Sigma_A^-$:
	\begin{align}\label{gradient g convex}
		\nabla_i g&=f'\left(\dfrac{\nabla_i \lambda_1}{\gamma} -\dfrac{\lambda_1}{\gamma^{2}}\nabla_i\gamma\right).
\end{align}
\begin{align*}
		L_\gamma g&=	\Delta_\gamma g+2\pI{\nabla g,\dfrac{\nabla\gamma}{\gamma}}_\gamma+\nabla_{n+1}g
		\\ \notag
		&=f''\norm{\nabla \left(\dfrac{\lambda_1}{\gamma}\right)}_\gamma^2 -\dfrac{f'}{\gamma}\left(\partial^2\gamma^{\:ab;cd}\nabla_1A_{ab}\nabla_1A_{cd}+2\dfrac{\norm{\nabla A_{1i}}_\gamma^2}{\lambda_i-\lambda_1}\right)
		\\\notag
		&\leq 0.
	\end{align*}

Next, we assume by contradiction that $\Sigma$ is not convex, i.e.,
	\begin{align*}
		I=\inf_\Sigma \dfrac{\lambda_1}{\gamma}\in \left[\dfrac{-1}{C(n-2)},0\right)\Leftrightarrow \inf_\Sigma g(p)\in [\beta, 0),\mbox{ where } \beta=f\left(I\right).
	\end{align*}

	\begin{lemma}\label{g interior minimum}
The function	$g$ cannot attain a negative minimum on $\Sigma^{-}$.
	\end{lemma}
	\begin{proof}
	Let assume by contradiction that $g$  attains its minimum at $p_0 \in \Sigma^-$. Then, since $\Sigma_A^-$ is dense in $\Sigma^-$, we have $L_\gamma(g) \leq 0$ on $\Sigma^-$. Therefore, the strong maximum principle implies that $g$ is constant on $\Sigma^-$.
		\\
		
	Moreover, we note that on any connected component of $\Sigma_A^{-}$, the following equation holds:
		\begin{align*}
			\norm{\nabla \left(\dfrac{\lambda_1}{\gamma}\right)}_\gamma^2=\partial^2\gamma^{\:ab;cd}\nabla_1A_{ab}\nabla_1A_{cd}=(1-\delta_{1i})\norm{\nabla A_{1i}}_\gamma^2=0.
		\end{align*}
	Then, since $\gamma(\lambda)$ is strictly increasing, we have at any fixed point $p \in \Sigma_A^-$ that the equations
		\begin{align*}
			0=\norm{\nabla A_{1i}}_\gamma^2=\partial\gamma^{ab}\nabla_aA_{1i}\nabla_bA_{1i}=\dfrac{\partial\gamma}{\partial\lambda_j}|\nabla_j A_{1i}|^2
		\end{align*}
	give that $|\nabla_j A_{1i}|(p) = 0$ for every $j$ and $i \neq 1$. Consequently, by the Codazzi equations, we obtain
		\begin{align*}
			0=\nabla_i\lambda_1=\nabla_1\lambda_i, \mbox{ for }i\neq 1. 
		\end{align*}
		
	Moreover, since $\gamma$ is strictly convex in off-radial directions, Lemma \ref{Eliptic estimate} implies 
	\begin{align*} 
	|\nabla_1 A_{ab}| = 0, \quad \text{for all } a, b. 
	\end{align*} 
Consequently, we have $|\nabla \lambda_1| = 0$ on every connected component of $\Sigma_A^{-}$. This implies that $\lambda_1$ is constant on the connected components of $\Sigma_A^{-}$, as is $\gamma$, since $\dfrac{\lambda_1}{\gamma}$ is a negative constant because $f'(r) > 0$. However, Eq. \eqref{gamma}, implies that $|A|_\gamma^2\gamma=0$, which is not possible since $\lambda_1 < 0 < \lambda_2$.
	\end{proof}	

	Next, we will use the Omori-Yau maximum principle on the sequence of translators given by $\Sigma_k = \set{ p + p_k : p \in \Sigma }$, where $p_k \in \Sigma$ is a sequence such that $\lim\limits_{k \to \infty} g(p_k) = \inf\limits_{\Sigma} g(p) = \beta < 0$. This means that we can assume, by going through another subsequence if necessary, that the sequence ${ p_k }$ satisfies the properties in Lemma \ref{Omori-Yau}.
	\newline
	
	Furthermore, since $|A|\leq H\leq \gamma\leq 1$ in $\Sigma$, by compactness we may go through another subsequence of $\Sigma_k$, if necessary, to obtain a smooth complete $\gamma$-translator $\Sigma_{\infty} $ containing the origin such that $\Sigma_k \to \Sigma_\infty$ in the smooth topology on compact subsets of $\Sigma_\infty$.
	\newline
	
	Let $\Sigma_\infty'$ be the connected component that contains the origin. We observe that at the origin of $\Sigma_\infty'$, $\lambda_1(0) = 0$, because if this were not the case, we would have $\inf\limits_\Sigma g(p) \in [\beta,0)$, and by Lemma \ref{Claim 1 Stric Concave}, we would conclude that $\Sigma_\infty'$ is the empty set, which is not possible. Furthermore, we note that $\gamma(\lambda(0)) = 0$, since 
	\begin{align*}
		\lim_{k\to\infty}g(p_k)\in [\beta,0)\Leftrightarrow\lim\limits_{k\to\infty}\dfrac{\lambda_1(p_k)}{\gamma(p_k)}\in\left(\dfrac{-1}{C(n-2)},0\right).		
	\end{align*}
Consequently, all the principal curvatures at the origin vanish, since $H \leq \gamma$ and $|A| \leq H$ on $\Sigma_{\infty}'$.
	\begin{lemma}\label{not at infinity convex}
The function $g(p)$ cannot be negative at $0 \in \Sigma_\infty'$.

	\end{lemma}
	
	\begin{proof}
	Firstly, we note that for a fixed point $p \in \Sigma_A^{-}$, Eq. \eqref{gradient g convex} can be expressed as
		\begin{align*}
			\nabla_i g=f'\left(\dfrac{\nabla_i\lambda_1}{\gamma}+ \dfrac{\lambda_1\lambda_i}{\gamma^2}\pI{\tau_i,e_{n+1}} \right).
		\end{align*}
		Moreover, we note that 
		\begin{align*}
-\dfrac{1}{C(n-2)}\leq			\lim_{k\to\infty}\dfrac{\lambda_1(p_k)}{\gamma(p_k)}<0<\lim_{k\to\infty}\dfrac{\lambda_2(p_k)}{\gamma(p_k)}\leq \ldots\leq \lim_{k\to\infty}\dfrac{\lambda_n(p_k)}{\gamma(p_k)}\leq 1.
		\end{align*}
	Then, by density and going through a diagonal subsequence, we may assume that $p_k\in\Sigma_A^{-}$, and therefore $\dfrac{\nabla_i\lambda_1(p_k)}{\gamma(p_k)}$ is bounded. 
		\newline
		
		Next, we will show that 
		\begin{align}\label{Last step g convex}
			\lim_{k\to\infty}	\dfrac{\nabla_i\lambda_1(p_k)}{\gamma(p_k)}=0, \forall i=1,\ldots,n.  
		\end{align}
	In this way, we would have
		\begin{align*}
			0=\lim_{k\to\infty}\pI{\tau_i(p_k),e_{n+1}}=\pI{\tau_i(0),e_{n+1}}.	
		\end{align*}
	Then, since $0=\gamma(0)=\pI{\nu(0),e_{n+1}}$, where $0\in\Sigma_{\infty}'$, we would obtain that $e_{n+1}=0$, which is not possible, finalizing the proof.  
		\newline
		
		Subsequently, to show \eqref{Last step g convex}, we will use the fact that 
		 $p_k$ satisfies the equations from Lemma \ref{Omori-Yau}. Therefore, the following holds at $p_k$:
		\begin{align*}
			\lim_{k\to\infty}L_\gamma g=&	\Delta_\gamma g+2\pI{\nabla g,\dfrac{\nabla\gamma}{\gamma}}_\gamma+\nabla_{n+1}g\geq \delta\in[0,\infty), 
		\end{align*}
		here we used that $\dfrac{\nabla \gamma}{\gamma}(p_k)$ is bounded.  
		\newline

		Furthermore, since 
		\begin{align*}
			L_\gamma g&=f''\norm{\nabla \left(\dfrac{\lambda_1}{\gamma}\right)}_\gamma^2 +\dfrac{f'}{\gamma}\left(-\partial^2\gamma^{\:ab;cd}\nabla_1A_{ab}\nabla_1A_{cd}-2\dfrac{\norm{\nabla A_{1i}}_\gamma^2}{\lambda_i-\lambda_1}\right)\leq 0,
		\end{align*}
		we note that $\lim\limits_{k\to\infty}L_\gamma g(p_k)=0$. 
		\newline
		Then, the term 
		\begin{align*}
			\restri{\left(\dfrac{\partial^2\gamma^{\:ab;cd}\nabla_1A_{ab}\nabla_1A_{cd}}{\gamma}+2\dfrac{\norm{\nabla A_{1i}}_\gamma^2}{\gamma(\lambda_i-\lambda_1)}\right)}{p_k}\to 0\mbox{ as }k\to \infty.
		\end{align*}
		In particular, since both terms in the above equation are positive at $p_k$,  we obtain 
		\begin{align*}
			\lim_{k\to\infty}\dfrac{\partial^2\gamma^{\:ab;cd}\nabla_1A_{ab}\nabla_1A_{cd}}{\gamma}=\lim_{k\to\infty}\dfrac{\norm{\nabla A_{1i}}_\gamma^2}{\gamma(\lambda_i-\lambda_1)}=0, \mbox{ for }i\neq 1. 
		\end{align*}
		
	In the following, each equation will be evaluated at 
	$p_k$. By Lemma \ref{Eliptic estimate}, it follows that
		\begin{align*}
			\dfrac{\partial^2\gamma^{\:ab;cd}\nabla_1A_{ab}\nabla_1A_{cd}}{\gamma}\geq \dfrac{|\nabla_1 A_{ab}|^2}{CH\gamma}\geq \left|\dfrac{\nabla_1 A_{ab}}{\gamma}\right|^2,\mbox{for all } a,b. 
		\end{align*}
Furthermore,  since $n\geq 3$, we have $\lambda_i-\lambda_1\leq H\leq \gamma$, and
		\begin{align*}
			\dfrac{\norm{\nabla A_{1i}}_\gamma^2}{\gamma(\lambda_i-\lambda_1)}\geq \dfrac{1}{C^2}\left|\dfrac{\nabla_iA_{1a}}{\gamma}\right|^2, \mbox{ for every }a\mbox{ and }i\neq 1.
		\end{align*}
Thus, we finally obtain that $\dfrac{\nabla \lambda_1}{\gamma}\to 0$ as $k\to\infty$, thereby  concluding the proof of Lemma \ref{not at infinity convex}.
	\end{proof}
	
Now, we continue with the case where  $\gamma$ is strictly concave in off-radial directions such that $\gamma\leq H$. For this purpose, we consider the function $\tilde{g}:\Sigma\to (-\infty,0)$ given by 
	\begin{align*}
		\tilde{g}(p)=\tilde{f}\left(\dfrac{\gamma(\lambda(p))}{H(p)-\lambda_1(p)}\right),
	\end{align*}
where $\tilde{f}(r)=f(r-1)$ and $f$ is given in Eq. \eqref{f}. We note that $\tilde{f}''<0$ and $\tilde{f}'>0$, and
both  are bounded away from $0$ on $[-r_2, -r_1]$, for any $0<r_1 < r_2<1$. 
\newline

	Firstly, since $\lambda\in\tilde{\Gamma}$, we observe that 
	\begin{align*}
		0<\dfrac{\gamma}{H-\lambda_1}\leq\dfrac{H}{H-\lambda_1}=1-\dfrac{\lambda_1}{H-\lambda_1}\leq2.  
	\end{align*}
	Therefore, we obtain $\tilde{g}(p)\in(\tilde{f}(0),0]$ for every $p\in\Sigma$. 	
	\newline

	Moreover, if $\Sigma$ were not convex, we would have that $g(p)$ is smooth on 
	\begin{align*}
		\Sigma^-=\set{p\in\Sigma: \tilde{g}(p)\in [\tilde{\beta}, \tilde{f}(1))}, 
	\end{align*}
	where $\tilde{\beta}=\tilde{f}(b) $, since $\lambda_1<0<\lambda_2$ and 
	\begin{align*}
		0<b=\dfrac{C^{-1}}{1+\frac{1}{n-2}}\leq\dfrac{\frac{\gamma}{H}}{1-\frac{\lambda_1}{H}}=\dfrac{\gamma}{H-\lambda_1}\leq \dfrac{\gamma}{H}\leq 1.
	\end{align*}

	Then, we have the following equations 
	\begin{align}\label{gradient g2}
		\nabla_i\tilde{g}&=\tilde{f}'\left(\dfrac{\nabla_i \gamma}{H-\lambda_1} -\dfrac{\gamma}{(H-\lambda_1)^{2}}\nabla_i(H-\lambda_1)\right),
	\end{align}
\begin{align*}	
		L_\gamma \tilde{g}&=	\Delta_\gamma \tilde{g}+2\pI{\nabla \tilde{g},\dfrac{\nabla(H-\lambda_1)}{H-\lambda_1}}_\gamma+\nabla_{n+1}\tilde{g}
		\\ \notag
		&=\tilde{f}''\norm{\nabla \left(\dfrac{\gamma}{H-\lambda_1}\right)}_\gamma^2 +\dfrac{\tilde{f}'(1-\delta_{1i})}{H-\lambda_1}\left(\partial^2\gamma^{\:ab;cd}\nabla_1A_{ab}\nabla_1A_{cd}-2\dfrac{\norm{\nabla A_{1i}}_\gamma^2}{\lambda_i-\lambda_1}\right)
		\\ \notag
		&\leq 0.
	\end{align*}
	
	Next, we will assume by contradiction that $\Sigma$ is not convex, i.e., 
	\begin{align*}
		J=\inf_\Sigma \dfrac{\gamma}{H-\lambda_1}\in [b,1)\Leftrightarrow \inf_\Sigma \tilde{g}(p)\in [\tilde{\beta}, 0).
	\end{align*}
	\begin{lemma}\label{Claim 1 Stric Concave}
	The function	$\tilde{g}$ cannot attain a negative minimum in $\Sigma^{-}$. 
	\end{lemma}
	\begin{proof}
		As in the proof of Lemma \ref{g interior minimum}, we assume by contradiction $\tilde{g}$ attains its minimum at $p_0\in\Sigma^-$. Then, since  $L_\gamma(\tilde{g})\leq 0$ on $\Sigma^-$, the strong maximum principle implies that $\tilde{g}$ is constant in $\Sigma^-$.
		\\
		
		Moreover, we note that in any connected component of $\Sigma_A^{-}$, the following equations holds
		\begin{align*}
			\norm{\nabla \left(\dfrac{\gamma}{H-\lambda_1}\right)}_\gamma^2=(1-\delta_{1i})\partial^2\gamma^{\:ab;cd}\nabla_iA_{ab}\nabla_iA_{cd}=(1-\delta_{1i})\norm{\nabla A_{1i}}_\gamma^2=0.
		\end{align*}
		Then, at any fixed point $p\in\Sigma_A^-$ the following equation
		\begin{align*}
			0=\norm{\nabla A_{1i}}_\gamma^2=\partial\gamma^{ab}\nabla_aA_{1i}\nabla_bA_{1i}=\dfrac{\partial\gamma}{\partial\lambda_j}|\nabla_j A_{1i}|^2
		\end{align*}
	which implies that $|\nabla_j A_{1i}|(p)=0$ for every $j$ and $i\neq 1$. Consequently,  by the Codazzi equations, we obtain 
		\begin{align*}
			0=\nabla_1\lambda_i=\nabla_i\lambda_1 \mbox{ for }i\neq 1. 
		\end{align*}
		
		In addition, since $\gamma$ is strictly concave in off-radial directions, we have by Lemma \ref{Eliptic estimate} that
		\begin{align*}
			|\nabla_iA_{ab}|=0, \mbox{for } i\neq 1. 
		\end{align*}
Particularly, by the Codazzi equation, we obtain
		\begin{align*}
			\nabla_i\lambda_{a}=0\mbox{ for }i\neq 1 \mbox{ and }a=1,\ldots,n. 
		\end{align*}
		Therefore, on every connected component of $\Sigma_A^{-}$,  we get  $|\nabla (H-\lambda_1)|=0$. This implies that $\gamma$ is constant in any connected components of $\Sigma_A^-$, since $\tilde{g}$ is constant in $\Sigma_A^{-}$ and $\tilde{g}'(r)>0$.  However, as in the proof of Lemma \ref{g interior minimum}, this cannot occurs.
	\end{proof}
	
	Next, as in the convex case, we will use the Omori-Yau maximum principle on the sequence of translators given by $\Sigma_k=\set{p+p_k:p\in\Sigma}$, where $p_k\in\Sigma$ is a sequence such that $\lim\limits_{k\to\infty}\tilde{g}(p_k)=\tilde{f}(J)<\tilde{f}(1)=0$.  
Then, by going through another subsequence if necessary, we have that there exists a smooth complete $\gamma$-translators  $\Sigma_{\infty }\subset\rr^{n+1}$ containing the origin such that $\Sigma_k\to\Sigma_\infty$ in the smooth topology on compact subsets of $\Sigma_\infty$.
	\newline
	
	Let $\Sigma_\infty'$ be the connected component that contains the origin of $\Sigma_\infty$. Then,  we observe that at the origin  of $\Sigma_\infty'$, $	\gamma(0)=\lim\limits_{k\to\infty}\gamma(p_k)=0$, because if it were not so, we would have  $\inf\limits_\Sigma \tilde{g}(p)\in (\tilde{\beta},0)$ and by Claim \ref{Claim 1 Stric Concave}, we would obtain that  $\Sigma_\infty'$ is the empty set, which is not possible. 
	\newline
	
	Furthermore, we also note that $H(0)=\lim\limits_{k\to\infty}H(p_k)= 0$ since $\gamma H^{-1}\in[C^{-1},C]$. In particular, $\lambda_i(0)=\lim\limits_{k\to\infty}\lambda_i(p_k)=0$ since $|A|\leq H$.   
	
	\begin{lemma}\label{not at infinity concave}
	The function $\tilde{g}(p)$ cannot be negative at $0\in\Sigma_{\infty}'$. 
	\end{lemma}
	
	\begin{proof}
		Firstly, we note that for a fixed point $p\in\Sigma_A^{-}$, Eq. \eqref{gradient g convex} can be expressed as  
		\begin{align*}
			\nabla_i \tilde{g}=\tilde{f}'\left(\dfrac{\lambda_i\pI{\tau_i,e_{n+1}}}{H-\lambda_1}+ \dfrac{\gamma\nabla(H-\lambda_1)}{(H-\lambda_1)^2} \right).
		\end{align*}
		Moreover, since $\lambda_2+\lambda_1>0$. at $p_k$, we have 
		\begin{align*}
		&\dfrac{-1}{n-1}\leq\dfrac{\lambda_1}{H-\lambda_1}<0,
		\\
		&0<1-\dfrac{\gamma}{H-\lambda_1}\leq 1-  \dfrac{H}{H-\lambda_1}=\dfrac{-\lambda_1}{H-\lambda_1}<\dfrac{\lambda_2}{H-\lambda_1}\leq \ldots\leq \dfrac{\lambda_n}{H-\lambda_1}\leq1.
		\end{align*}
	Then,  since $J<1$, the limits $\lim\limits_{k\to\infty}\dfrac{\lambda_i}{H-\lambda_1}$ are all bounded away from zero. Therefore, by density, we may assume that  $p_k\in\Sigma_A^{-}$, and hence $\dfrac{\nabla\gamma}{H-\lambda_1}$ is uniformly bounded at $p_k$. 
		\newline
		
		Considering the above, we will show that 
		\begin{align}\label{Last step g concave}
			\lim_{k\to\infty}	\dfrac{\nabla_i(H(p_k)-\lambda_1(p_k))}{H(p_k)-\lambda_1\gamma(p_k)}=0, \forall i=1,\ldots,n.  
		\end{align}
		In this way, we would obtain the same contradiction as in the convex case, finalizing the proof.  
		\newline
		
		Subsequently, we observe that Lemma \eqref{Omori-Yau} implies
		\begin{align*}
			\lim_{k\to\infty}L_\gamma \tilde{g}=&	\Delta_\gamma \tilde{g}+2\pI{\nabla \tilde{g},\dfrac{\nabla\gamma}{\gamma}}_\gamma+\nabla_{n+1}\tilde{g}\geq \delta\in[0,\infty).
		\end{align*}
	On the other hand, 
		\begin{align*}
			L_\gamma \tilde{g}&=\tilde{f}''\norm{\nabla \left(\dfrac{\gamma}{H-\lambda_1}\right)}_\gamma^2 +\dfrac{\tilde{f}'(1-\delta_{1i})}{H-\lambda_1}\left(\partial^2\gamma^{\:ab;cd}\nabla_iA_{ab}\nabla_iA_{cd}-2\dfrac{\norm{\nabla A_{1i}}_\gamma^2}{\lambda_i-\lambda_1}\right)
			\\
			&\leq 0.
		\end{align*}
Therefore, we obtain  $\lim\limits_{k\to\infty}L_\gamma \tilde{g}(p_k)=0$. 
\newline

Then, the term 
		\begin{align*}
			\restri{\left(\dfrac{\partial^2\gamma^{\:ab;cd}\nabla_iA_{ab}\nabla_iA_{cd}}{H-\lambda_1}-\dfrac{2\norm{\nabla A_{1i}}_\gamma^2}{(H-\lambda_1)(\lambda_i-\lambda_1)}\right)}{p_k}\to 0\mbox{ as }k\to \infty, \mbox{ for }i\neq 1.
		\end{align*}
	In particular, since both terms in the above equation are non-positive, we have
		\begin{align*}
			\lim_{k\to\infty}\dfrac{\partial^2\gamma^{\:ab;cd}\nabla_1A_{ab}\nabla_1A_{cd}}{H-\lambda_1}=\lim_{k\to\infty}\dfrac{\norm{\nabla A_{1i}}_\gamma^2}{(H-\lambda_1)(\lambda_i-\lambda_1)}=0, \mbox{ for }i\neq 1. 
		\end{align*}
Furthermore,  by Lemma \ref{Eliptic estimate}, we observe that
		\begin{align*}
			-\dfrac{\partial^2\gamma^{\:ab;cd}\nabla_iA_{ab}\nabla_iA_{cd}}{H-\lambda_1}\geq \dfrac{|\nabla_i A_{ab}|^2}{H(H-\lambda_1)}\geq \left|\dfrac{\nabla_i A_{ab}}{H-\lambda_1}\right|^2,
		\end{align*}
		for all $a,b$ and $i\neq 1$.  Moreover, since 
		\begin{align*}
			\dfrac{\norm{\nabla A_{1i}}_\gamma^2}{(H-\lambda_1)(\lambda_i-\lambda_1)} \geq \dfrac{1}{C^2}\left|\dfrac{\nabla_1 A_{ai}}{H-\lambda_1}\right|^2, \mbox{ for every } a \mbox{ and } i \neq 1,
		\end{align*}
		we finally obtain that \( \dfrac{\nabla (H-\lambda_1)}{H-\lambda_1} \to 0 \) as \( k \to \infty \), finishing Lemma \ref{not at infinity concave}.
		\end{proof}

	\section{\textbf{Grim-reaper cylinders}}\label{Grim}\,
	We recall from the introduction that for a given non-degenerate normalized curvature function \( \gamma: \overline{\Gamma} \to [0, \infty) \), i.e., \( \gamma(1, 0, \ldots, 0) = 1 \), a Grim Reaper cylinder is a $\gamma$-translator isometric to \( \mathbb{R}^n \times G^\pi \), where
	\[
	G^\pi = \left\{ (x, -\ln(\sin(x))) : x \in (0, \pi) \right\} \subset \mathbb{R}^2
	\]
	is the Grim Reaper curve. More precisely, a Grim Reaper cylinder is the graph of a function defined on a slab of width \( \omega \geq \pi \), given by 	$u_\omega: \left(0, \omega\right) \times \mathbb{R}^{n-1} \to \mathbb{R}$
	\begin{align*}
	u_\omega(x) = -\left(\frac{\omega}{\pi}\right)^2 \ln\left(\sin\left(\frac{\pi}{\omega} x_1\right)\right) + x_{n} \sqrt{\left(\frac{\omega}{\pi}\right)^2 - 1}.
	\end{align*}
In fact, by setting \( G^\omega = \text{Graph}(u_\omega) \), the coefficients of the shape operator of \( G^\omega \) with respect to the outward unit normal vector 
are given by \( A^i_j = \frac{\pi}{\omega} \sin\left(\frac{\pi}{\omega} x_1\right) \) for \( i = j = 1 \), and \( A^i_j = 0 \) otherwise. Therefore, the principal curvatures of \( G^\omega \) are
\begin{align*}
	\lambda_n = \ldots = \lambda_2 = 0 \quad \text{and} \quad \lambda_1 = \frac{\pi}{\omega} \sin\left(\frac{\pi}{\omega} x_1\right).
\end{align*}
Consequently, since \( \gamma \) is normalized, we have
\begin{align*}
	\gamma(\lambda) = \lambda_1 \gamma(1, 0, \ldots, 0) = \langle \nu, e_{n+1} \rangle \quad \text{on } G^\omega,
\end{align*}
implying that \( G^\omega \) is a convex \( \gamma \)-translator in \( \mathbb{R}^{n+1} \).
	\newline
	
	Due to the following maximum principle proved in \cite[Theorem 1.2]{Yo}: 	
	\begin{theorem}\label{Tangency}
		Let $\Sigma_1,\Sigma_2\subset\rr^{n+1}$ be two connected $\gamma$-translators  such that
		\begin{enumerate}
			\item $\gamma:\Gamma\to[0,\infty)$ satisfies properties \ref{a)}-\ref{c)} on $\Gamma$.
			\item $\Sigma_1$ is strictly convex (i.e.: $\lambda\in\Gamma_+$ on $\Sigma_1$).
			\item $\Sigma_2$ is convex (i.e.: $\lambda\in\overline{\Gamma}_+$ in $\Sigma_2$) not totally geodesic.
		\end{enumerate}
		Then,
		\begin{enumerate}[a)]
			\item (\textbf{Interior tangential principle}) Assume that there exists an interior point $p\in \Sigma_1\cap \Sigma_2$ such that the tangent spaces coincide at $p$. If $\Sigma_1$ lies at one side of $\Sigma_2$, then both hypersurfaces coincide.
			
			\item (\textbf{Boundary tangential principle}) Assume that the boundaries $\partial \Sigma_i$ lie in the same hyperplane $\Pi$ and the intersection of $\Sigma_i$ with $\Pi$ is transversal. If $\Sigma_1$ lies at one side of $\Sigma_2$ and there exist $p\in \partial \Sigma_1\cap \partial \Sigma_2$ such that the tangent spaces to $\Sigma_i$ and $\partial\Sigma_i$ coincide, then both hypersurfaces coincide. 
		\end{enumerate}
	\end{theorem}

Now, we give prove Thm. \ref{non-existence} restated here.
			\begin{figure}[t]
		\centering
		\includegraphics[width=0.4\textwidth]{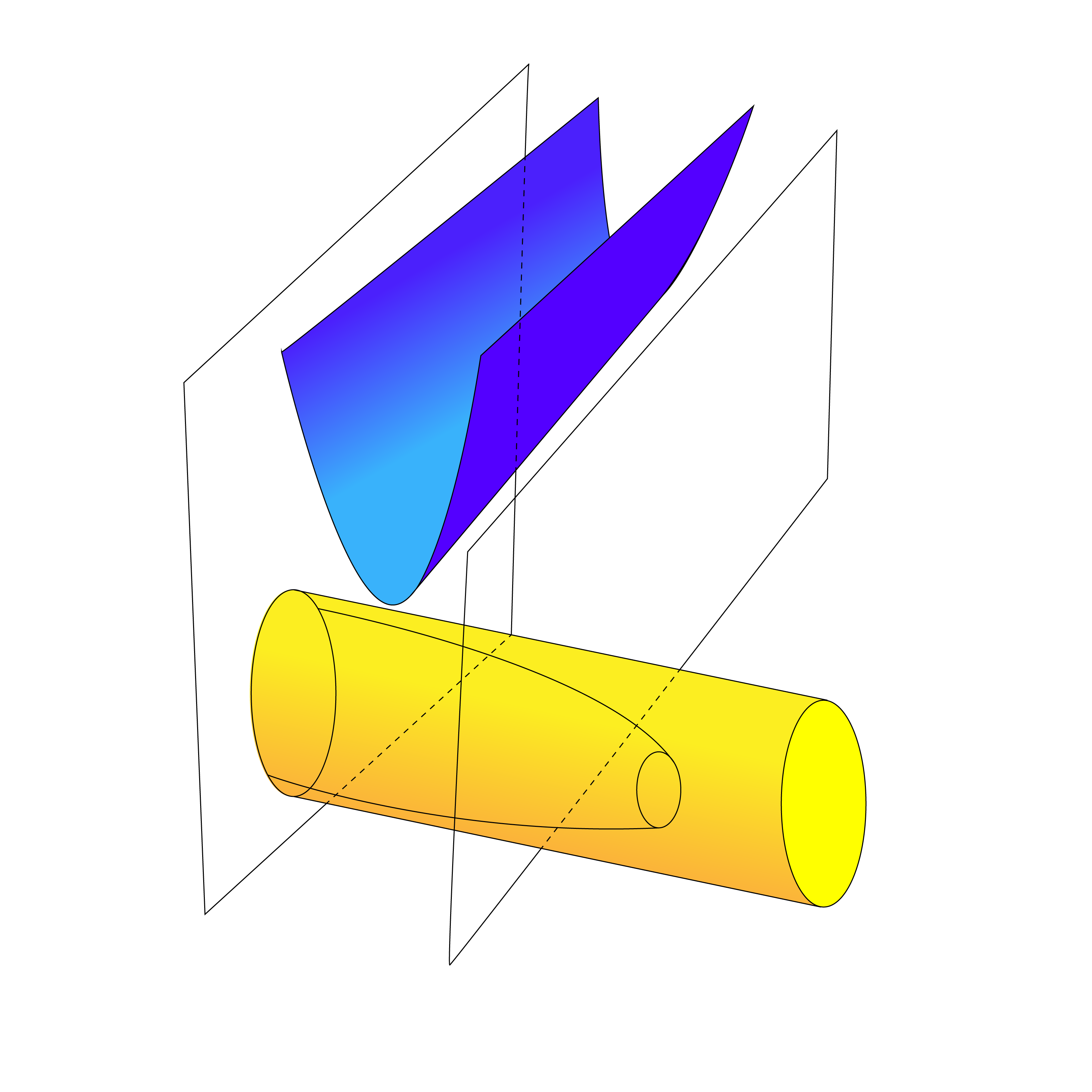}
		\caption{Theorem \ref{non-existence} configuration.}
		\label{fig:mi_imagen1}
	\end{figure}

	\begin{theorem}[Thm. \ref{non-existence}]
Let \( \gamma: \overline{\Gamma} \to [0, \infty) \) be a non-degenerate, normalized curvature function. Then, there does not exist a complete, strictly convex \( \gamma \)-translator \( \Sigma \subset \mathbb{R}^{n+1} \) with compact boundary contained in any non-vertical cylinder.
	\end{theorem} 
	
	\begin{proof}
	We argue by contradiction, assuming the existence of a complete $\gamma$-translator $\Sigma \subset \mathbb{R}^{n+1}$ with a compact boundary contained within a round non-vertical cylinder $C_r$ of radius $r > 0$.
	\newline
	
	Then, due to the compactness of $\partial \Sigma$, there exists $\theta_0 \in (0, \frac{\pi}{2})$ such that the slab region $S = \{ x \in \mathbb{R}^{n+1} : |\langle x, e_1 \rangle| \leq \pi \sec{\theta_0} \}$
	satisfies $S \cap \partial \Sigma = \emptyset$, see Fig. \ref{fig:mi_imagen1}.
	\newline
	
	Next, we consider the Grim Reaper cylinder $G^\pi$ contained in $S$. There exists $t_0 \in \mathbb{R}$ such that $G^\pi + t e_{n+1}$ does not intersect $\Sigma$ for $t \geq t_0$. We then translate $G^\pi + t_0 e_{n+1}$ downward until it first contacts $\Sigma$. This is possible since $S \cap \Sigma$ is compact, being contained within the intersection of $S$ and a non-vertical cylinder.
\newline
	
	Finally, because $S\cap \partial \Sigma = \emptyset$, the contact point is an interior point. Consequently, by Theorem \ref{Tangency}, we conclude that $\Sigma = G^\pi$, contradicting the assumption that $\Sigma$ is contained in a round cylinder.
	
	\end{proof}
	
	\begin{figure}[h]
		\centering
		\includegraphics[width=.4\textwidth, height=2.5cm]{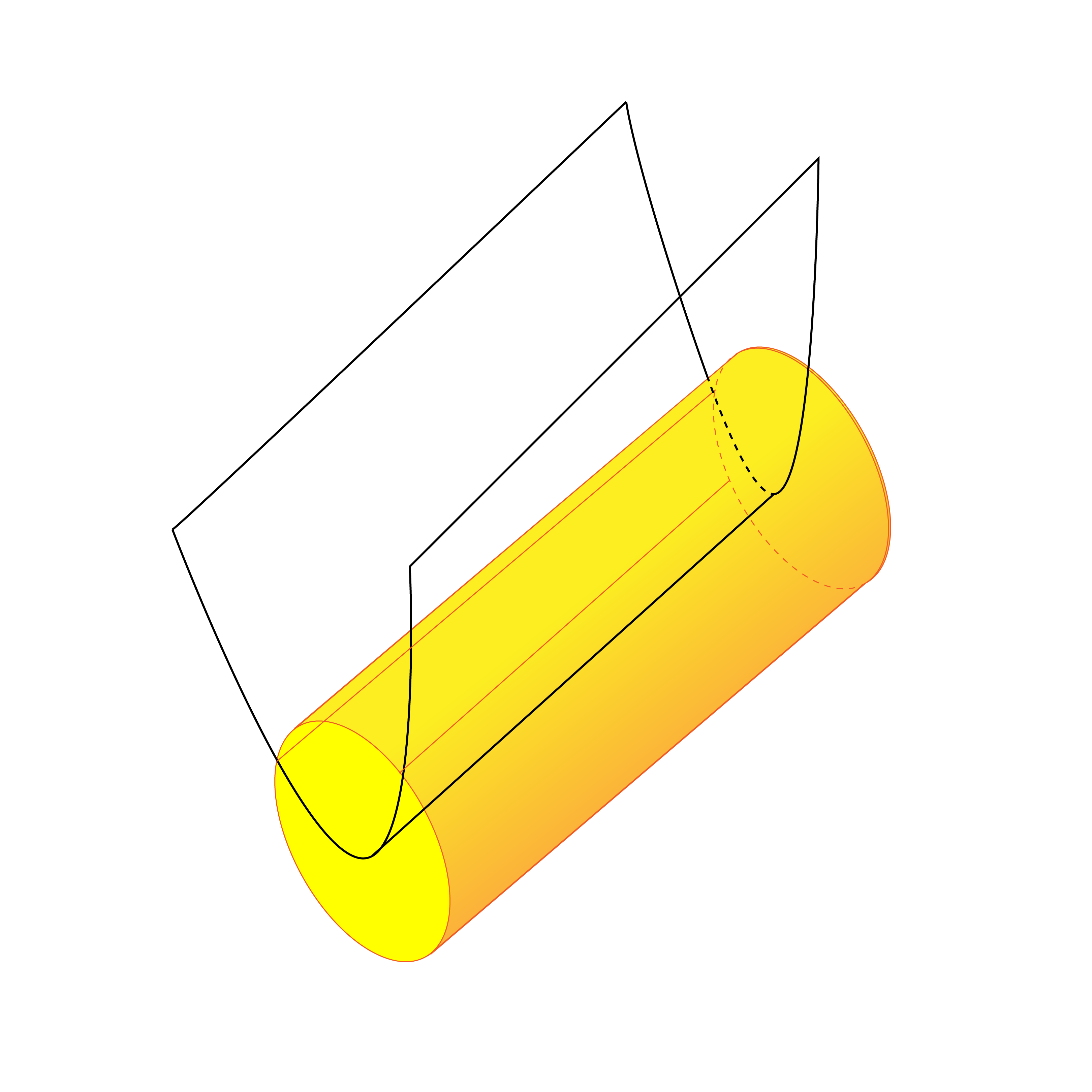}
		\caption{Intersection of a grim reaper cylinder with round cylinder.}
		\label{fig:mi_imagen}
	\end{figure}
	\begin{remark}
We note that there are counterexamples in which the assumption on the boundary of $\Sigma$ is not required. For instance, consider the intersection of $G^\omega$ with a round cylinder whose axis is the straight line $\left(t, 0, \ldots, u_\omega(0, \ldots, 0, t)\right)$, see Fig. \ref{fig:mi_imagen}.
	\end{remark}
		
	\begin{proposition}\label{prop}
		Let $\gamma:\overline{\Gamma}\to[0,\infty)$ be a non-degenerate normalized curvature function. Let $\Sigma\subset\rr^{n+1}$ be a complete convex $\gamma$-translator such that the principal curvatures of $\Sigma$ are all zero except of one. Then $\Sigma$ is a Grim Reaper cylinder for some width $\omega\geq\pi$ up to a vertical translation. 
	\end{proposition}
	
	\begin{proof}
		
		Firstly, we  rearrange the principal directions to be $\vec{e}_1$ with principal curvature $\lambda_1>0$. Then, by applying further rigid motions, we may assume that $\Sigma$ is a vertical graph of the form $	x_{n+1}=u(x_1)+ax_{n}$ for some smooth function $u:I\subset\rr\to\rr$ and $a\geq 0$.  For a proof of this fact we  refer the reader to \cite{hartman1959spherical}. 
		\newline
		
		Subsequently, the shape operator with respect the ouward unit normal
		is given by $A^i_j=
				\left(1-\frac{u'}{1+a^2+(u')^2}\right)\frac{u''}{\sqrt{1+a^2+(u')^2}}$ for $ i=j=1$ and $0$ otherwise. Therefore, Eq. \eqref{gamma-trans} gives an initial value problem of the form
		\begin{align*}
			\begin{cases}
				(1+a^2)u''=1+a^2+(u')^2, 
				\\	
				u'(0)=0, u(0)=h.
			\end{cases}
		\end{align*}
		for some height $h\in\rr$. Finally, by taking $\tilde{u}(x)=b^2u\left(xb^{-1}\right)$ with $b^2=1+a^2$, we obtain
		\begin{align*}
			\tilde{u}''=1+(\tilde{u}')^2.
		\end{align*}
		This means that $\tilde{u}=h-\ln(\cos(x))$ with $x\in\left(0,\pi\right)$. Recall that the width of $\Sigma$ is defined by the relation $\omega\pi^{-1}=b$. 
	\end{proof}
	
	The following result is a characterization of $G^\omega$ in terms of angle functions of $\Sigma$.
	
	\begin{theorem}
	Let $\gamma:\Gamma\to (0,\infty)$ be a non-degenerate normalized  curvature function. Let $\Sigma\subset\rr^{n+1}$ be a complete convex not totally geodesic $\gamma$-translator. Then $\Sigma$ is, up to vertical translation, a Grim Reaper cylinder for some width $\omega\geq \pi$ if, and only if, the angle functions $\pI{\nu,e_i}$ vanish identically on $\Sigma$ for $i=2,\ldots,n$ .
	\end{theorem}
	
	\begin{proof}
		Let us assume that the angle functions $\theta_i=\pI{\nu,e_i}$ for $i=2,\ldots,n,$ vanish identically on $\Sigma$. Then, by choosing a orthonormal frame $\set{\tau_i}\subset T_p\Sigma$ of principal directions at $p_0$, we have 
		\begin{align*}
			\nabla_a\theta_i&=\pI{\nabla_a\nu, e_i}=A_{ac}\pI{\tau_c,e_i},
		\end{align*}
\begin{align*}
			\Delta_\gamma\theta_i&=\dfrac{\partial\gamma}{\partial A_{ab}}\nabla_b\left(\nabla_a\theta_i\right)=\dfrac{\partial\gamma}{\partial A_{ab}}\left(\left(\nabla_bA_{ac}\right)\pI{\tau_c,e_i}+A_{ac}\pI{\nabla_b\tau_c,e_i}\right)
			\\
			&=\dfrac{\partial\gamma}{\partial A_{ab}}\left(\left(\nabla_cA_{ab}\right)\pI{\tau_c,e_i}-A_{ac}A_{bc}\theta_i\right)
			\\
			&=\pI{\nabla\gamma,e_i}-|A|_\gamma^2\theta_i. 
		\end{align*}
		Therefore, since $\theta_i=0$ on $\Sigma$ for $i=2,\ldots,n$, we have  $\pI{\nabla\gamma,e_i}=0$ on $\Sigma$ for $i>1$. We note that this is equivalent to say that  $\gamma$ only depends on one principal curvature of $\Sigma$, because 
		\begin{align*}
			0=\pI{\nabla\gamma,e_i}=\pI{D\gamma,e_i}-\pI{D\gamma,\nu}\theta_i=D_i\gamma
		\end{align*}
		on $\Sigma$ for $i=2,\ldots,n$, where $D$ denotes the euclidean connection of $\rr^{n}$. 
		\newline
		
		Consequently, only one principal curvature, say $\lambda_1>0$, does not vanish identically on $\Sigma$. To see this, we note that the $1$-homogeneity allow us to write 
		\begin{align*}
			\lambda_1=\lambda_1\gamma(1,0,\ldots,0,)=\gamma(\lambda_1,0,\ldots,0)=\partial_1\gamma(\lambda_1,0,\ldots,0)\lambda_1,
		\end{align*}
		which gives $1=\partial_1\gamma(\lambda_1,0,\ldots,0)=\partial_1\gamma(\lambda)$ on $\Sigma$. Then, we note that 
		\begin{align*}
			\lambda_1+\partial_2\gamma(\lambda)\lambda_2\ldots+\partial_n\gamma(\lambda)\lambda_n	=\gamma(\lambda)=\gamma(\lambda_1,0,\ldots,0)=\lambda_1.
		\end{align*}
		This implies that $\lambda_i=0$ for $i=2,\ldots,n$ since $\gamma$ is increasing in each variable and $\Sigma$ is  convex. Therefore, the results follows by Proposition \ref{prop}. 
	\end{proof}
	
Now, we give the proof of Thm. \ref{convex}, restated here. 
	
	\begin{theorem}[Thm. \ref{convex}]
	Let  $\gamma:\overline{\Gamma}\to[0,\infty)$ be a convex non-degenerate normalized curvature function.
	Let $\Sigma\subset\rr^{n+1}$ be a complete convex non-totally geodesic $\gamma$-translator such that there exists a point $p_0\in\Sigma$ that satisfies  
	\begin{align*}
\lambda_{2}(p_0)=\ldots=\lambda_{n}(p_0)=0,		
	\end{align*}
 where the principal curvatures of $\Sigma$ are ordered by $\lambda_n\leq\ldots\leq\lambda_1$. Then, $\Sigma$ is a Grim Reaper cylinder. 
In particular, if $\Sigma$ contains a straight line, then $\Sigma$ is a Grim Reaper cylinder. 
	\end{theorem}
	
	\begin{proof}
		Let $\set{\tau_1,\ldots,\tau_n}$ be an orthonormal frame of principal directions at $p_0\in\Sigma$. Then, by ordering the principal curvatures  by $\lambda_n\leq\ldots\leq\lambda_1$, it follows that $\lambda_i(p_0)=A_{ii}(p_0)=0$ for $i=2,\ldots,n$, and the entries $A_{ii}$ reach a local minimum at $p_0$ for $i=2,\ldots,n-1$.
		\newline
		
		On the other hand, we have Eq. \eqref{A_ij} 
		\begin{align*}
			\Delta_\gamma A_{ii}+\dfrac{\partial^2\gamma}{\partial A_{ab}\partial A_{cd}}\nabla_iA_{ab}\nabla_i A_{cd}+|A|_\gamma^2A_{ii}+\nabla_{n+1} A_{ii}=0.
		\end{align*}
		Then, the strong maximum principle applied to  $A_{ii}$ implies that  $\lambda_i$ vanish identically on $\Sigma$ for $i=2,\ldots,n-1$. Moreover, since $\sup\limits_{p\in\Sigma}\lambda_n(p)=\lambda_n(p_0)=0$, the principal curvatures of in $\Sigma$ are $(\lambda_1,0,\ldots,0)$ and the result follows by proposition \ref{prop}. 
	\end{proof}
	
Next, we give the proof of Thm. \ref{T1}, restated here. 
	\begin{theorem}[Thm. \ref{T1}]
		Let $\gamma:\Gamma\to (0,\infty)$ be a  non-degenerate normalized concave curvature function.  Let $\Sigma$ be a complete convex not totally geodesic $\gamma$-translator. Then, $\Sigma$ is a grim reaper cylinder if, and only if, the function $|A|^2\gamma^{-2}$ attains a local maximum on $\Sigma$.
	\end{theorem}
	
	\begin{proof}
		We consider $f(p)=|A|^2\gamma^{-2}$. Then, by using an orthonormal frame $\set{\tau_i}\subset T_p\Sigma$ of principal directions, we have the following equations at  $p\in\Sigma$:
		\begin{align*}
			&\nabla f=\dfrac{\nabla |A|^2}{\gamma^2}-2f\dfrac{\nabla \gamma}{\gamma},
			\\
			&\Delta_\gamma f=\dfrac{\Delta_\gamma |A|^2}{\gamma^2}-2\dfrac{f}{\gamma}\Delta_\gamma\gamma-4\pI{\nabla f,\dfrac{\nabla\gamma}{\gamma}}_{\gamma}-2\dfrac{f}{\gamma^2}\norm{\nabla \gamma}_\gamma^2.
		\end{align*}
		Furthermore, by Proposition \ref{equations}, we observe that
		\begin{align}\label{Eq f}
			\Delta_\gamma f+2\dfrac{\lambda_{i}}{\gamma^2}\dfrac{\partial^2\gamma}{\partial A_{ab}A_{cd}}\nabla_iA_{ab}\nabla_iA_{cd}+\nabla_{n+1} f+2\pI{\nabla f,\dfrac{\nabla \gamma}{\gamma}}_{\gamma}=2\gamma^{-4}Q^2,
		\end{align}
		where 
		\begin{align*}
			Q^2&=\gamma^2\norm{\nabla A}_\gamma^2+|A|^2\norm{\nabla \gamma}_\gamma^2-\gamma\pI{\nabla |A|^2,\nabla \gamma}_\gamma
			\\
			&=\norm{\gamma\nabla A-\nabla\gamma\otimes A}_\gamma^2\geq 0. 
		\end{align*}
		
		Then, the maximum principle applied to the Equation \eqref{Eq f} implies that $f(p)$ is constant on $\Sigma$, since the term   
		\begin{align*}
			\dfrac{\lambda_i}{\gamma^2}\dfrac{\partial^2\gamma}{\partial A_{ab}A_{cd}}\nabla_iA_{ab}\nabla_iA_{cd}\leq 0,
		\end{align*}
		by the concavity of $\gamma$ and the convexity of $\Sigma$.  
		\newline
		
		In particular, the term $\gamma^{-4}Q^2$ in Equation \eqref{Eq f} must vanishes identically on $\Sigma$. This means that 
		\begin{align}\label{Condition 0}
			\gamma\left(\nabla_i A(\tau_k,\tau_j)\right)-\nabla_i \gamma A(\tau_k,\tau_j)=0, \mbox{ for all }i,,j,k.
		\end{align}
		Furthermore,  by the Codazzi equations we have
		\begin{align}\label{Condition}
			(\nabla_i\gamma )A(\tau_j,\tau_k)=(\nabla_j\gamma)A(\tau_i,\tau_k) ,\mbox{ for all }i,j,k.
		\end{align}
		
		Now, we analyze two cases from  Equation \eqref{Condition}.
		\newline
		\textbf{Case 1:} Let us assume that $\gamma$ is a positive constant. Then, Eq. \eqref{gamma} implies that $|A|_\gamma^2\gamma=0$. But since condition \ref{a)} and \ref{c)} holds, it follows that $\lambda_i=0$ for all $i$, but this contradicts that $\Sigma$ is not totally geodesic.
		\newline
		\newline
		\textbf{Case 2:} Now we choose a simply connected neighborhood $U$ such that $|\nabla \gamma|\neq 0$. Then, for each point $p\in U$ we choose an orthonormal frame on $T_pU$ given by
		\begin{align*}
			\set{\tau_1=\dfrac{\nabla \gamma}{|\nabla \gamma|},\tau_2,\ldots,\tau_n}. 
		\end{align*} 
		We note that in this frame $\nabla_j\gamma=0$ and $\nabla_1\gamma=|\nabla\gamma|$. Moreover,  Equation \eqref{Condition} implies that $A(\tau_j,\tau_k)=0$ for $k\geq 1$ and $j\geq 2$.
		Consequently, we have $\Sigma$ posses only one nonzero principal curvature on $U$. 
		\newline
		
		Unfortunately, we cannot apply Proposition \ref{prop} because $U$ is not necessarily complete. However, we consider the nullity distribution $\mathcal{D}:\Sigma\to T\Sigma$ associated to the second fundamental form $A$, i.e:
		\begin{align*}
			\mathcal{D}(p)=\set{v\in T_p\Sigma: A(v,\cdot)=0}. 
		\end{align*}
		It is a well know fact that $\mathcal{D}$ is smooth, integrable and  autoparallel, i.e: for any $X, Y\in\mathcal{D}$ we have $\nabla_XY\in \mathcal{D}$. The integral submanifolds of $\mathcal{D}$ are totally geodesic in $\Sigma$ and their images via the immersion $F:\Sigma\to\rr^{n+1}$ are totally geodesic submanifolds of $\rr^{n+1}$. Moreover, the gauss map $\nu:\Sigma\to\Sp^n$ of the immersion $F$ is constant along the leaves of $\mathcal{D}$. 
		\newline
		
		Now, let $\mathcal{D}^{\perp}=\mbox{span}(\tau_i)$ be the orthogonal complement of $\mathcal{D}$ and we fix $p_0\in\Sigma$ and set $\tau=\tau_1(p_0)$. Then, we extend $\tau$ by parallel transport to a vector field along the integral curve $\alpha:(-\eps,\eps)\to\Sigma$
		with direction of $\tau_1$ passing through $p_0$. 
		\newline
		
		Then, due to Eq. \eqref{Condition 0}, the shape operator $\W:T\Sigma\to T\Sigma$ satisfies 
		\begin{align*}
			(\nabla _{\tau_1}\W)v=\tau_1(\gamma)v,
		\end{align*}
		for any $v\in T\Sigma$. In particular,  we have 
		\begin{align*}
			\dfrac{d}{dt}|\left(\W-\gamma\mbox{Id}\right)\tau|^2&= 2\pI{\left(\nabla_{\tau_1}\W\right)\tau- \tau_1(\gamma)\tau, \W\tau-\gamma\tau}=0.
		\end{align*}
		Therefore $\tau=\tau_1$ is a parallel vector field along the integral curves of
		$\tau_1$. Then, since $\mathcal{D}^{\perp}$ is spanned by $\tau_1$, we obtain that
		the orthogonal complement $\mathcal{D}^\perp$ of $\mathcal{D}$ is autoparallel in $T\Sigma$. 
		\newline
		
		Subsequently, since $T\Sigma=\mathcal{D}\oplus\mathcal{D}^{\perp}$ and trivially $\mathcal{D}$ and $\mathcal{D}^\perp$ are parallel to each other, the de Rham decomposition Theorem give us that $F(U)$ splits as the Cartesian product of a plane curve $G$ with a euclidean
		factor $\rr^n$. Then, since $\Sigma$ is $\gamma$-translator this curve must be a Grim Reaper curve. 
		\newline
		
		Consequently, $F(\Sigma)$ contains an open neighborhood that is part of a Grim Reaper cylinder, and by Theorem \ref{Tangency}, we finally obtain that $\Sigma$ is a Grim Reaper cylinder.
	\end{proof}
	
	\section{\textbf{Conclusions}}\label{Conclusion}
In summary, in this paper, we used maximum principle-type arguments to generalize well-known results on the characterization scheme used for $H$-translators in $\mathbb{R}^{n+1}$ for fully nonlinear curvature functions 
$\gamma: \overline{\Gamma} \to [0, \infty)$ with a sign on $\text{Hess}(\gamma)$. 
\newline

On one hand, as a corollary of the main result and the characterization of Grim Reaper cylinders obtained in this paper we can give the proof  of Corollary \ref{Coro}:
\begin{proof}
Let $\Sigma \subset \mathbb{R}^{n+1}$ be a $\gamma$-translator with $\gamma: \overline{\Gamma} \to [0, \infty)$ being a normalized non-degenerate curvature function as in Thm.\ref{Generalization of SX}. Then, if $\Sigma$ is not strictly convex, it holds $\Sigma=\Sigma'\times\rr^{n-k}$ where $\Sigma'$ is strictly convex graph in $\rr^{k+1}$: in the case of $\gamma$ is convex, there exists a point $p_0 \in \Sigma$ such that $\lambda_1(p_0) = 0 < \lambda_2(p_0)$. Then, by Theorem \ref{convex}, $\Sigma$ is a Grim Reaper cylinder. On the other hand, if  $\gamma$ concave, then $f(p)=\dfrac{|A|^2}{\gamma}$ is a $0$-homogeneous function such that $\partial_i f(1,0,\ldots,0)=-2\partial_i\gamma(1,0,\ldots,0)$, if $i\neq 1$ and $\partial_i f(1,0,\ldots,0)=0$, otherwise. Therefore, $f(p)$ attains a local maximum at $p_0\in\Sigma$ such that $\lambda(p_0)=(1,0,\ldots,0)\in\Gamma$. Then, by Thm. \ref{T1}, $\Sigma$ is a Grim Reaper cylinder. 
\end{proof}
	Consequently, the possible graphical solutions to Eq. \eqref{gamma-trans} in $\mathbb{R}^{3}$ under the above assumptions are the Grim Reaper cylinders and the strictly convex solutions, including among them the ``bowl''-type solutions constructed in \cite{Shati}. 

 Finally, we would like to propose the following open questions related to the characterization of convex $\gamma$-translators in $\mathbb{R}^{n+1}$: 
 \newline	
 \textbf{Q1:} Let $\Sigma = \set{(x, u(x)) \in \mathbb{R}^{n+1}:x\in\rr^n}$ be a strictly convex, complete $\gamma$-translator. Under which geometric constraints does the following growth estimate 
 	 \begin{align*}
 	 	 u(x) \leq C|x|^{\alpha+1}, \quad \text{for } |x| \geq R \gg 1, \mbox{ hold ?}
  	  \end{align*} 
    This result is well-known for $\gamma = H^{\alpha}$, with the case $\alpha = 1$ first proven in \cite{wang2011convex}, and the case $\alpha \neq 1$ in \cite{chen2015convex}. 
 \newline
 \textbf{Q2:}What class of curvature functions  allows the existence of the family of $\Delta$-wing solutions? These solutions are well-known for $\gamma = H$, see \cite{hoffman2019graphical}.

\end{document}